\documentclass[12pt]{amsart}
\usepackage[shortlabels]{enumitem}
\usepackage{amssymb}
\usepackage{mathrsfs}
\usepackage{graphicx}
\usepackage{stmaryrd}
\usepackage[all]{xy}
\usepackage[margin=1in]{geometry} 
\usepackage[bookmarks, bookmarksdepth=2, colorlinks=true, linkcolor=blue, citecolor=blue, urlcolor=blue]{hyperref}
\usepackage{eucal}
\usepackage{contour}
\usepackage[normalem]{ulem}
\usepackage{tikz}
\usetikzlibrary{positioning}
\usetikzlibrary{arrows.meta}
\usetikzlibrary{decorations.markings}
\usetikzlibrary{plotmarks}
\usetikzlibrary{shapes.misc}
\usetikzlibrary{shapes.geometric}
\usetikzlibrary{backgrounds}
\usepackage{dsfont}

\numberwithin{equation}{section}
\newtheorem{theorem}[equation]{Theorem}

\newtheorem{proposition}[equation]{Proposition}
\newtheorem{lemma}[equation]{Lemma}
\newtheorem{corollary}[equation]{Corollary}

\theoremstyle{definition}
\newtheorem{rmk}[equation]{Remark}
\newenvironment{remark}[1][]{\begin{rmk}[#1] \pushQED{\qed}}{\popQED \end{rmk}}
\newtheorem{eg}[equation]{Example}
\newenvironment{example}[1][]{\begin{eg}[#1] \pushQED{\qed}}{\popQED \end{eg}}
\newtheorem{defnaux}[equation]{Definition}
\newenvironment{definition}[1][]{\begin{defnaux}[#1]\pushQED{\qed}}{\popQED \end{defnaux}}

\newcommand{\cC}{\mathcal{C}}

\newcommand{\cD}{\mathcal{D}}

\newcommand{\sE}{\mathscr{E}}
\newcommand{\bF}{\mathbf{F}}

\newcommand{\fF}{\mathfrak{F}}

\newcommand{\bN}{\mathbf{N}}

\newcommand{\cO}{\mathcal{O}}

\newcommand{\bQ}{\mathbf{Q}}

\newcommand{\fS}{\mathfrak{S}}

\newcommand{\bT}{\mathbf{T}}

\newcommand{\fT}{\mathfrak{T}}

\newcommand{\cX}{\mathcal{X}}

\newcommand{\cY}{\mathcal{Y}}

\newcommand{\bZ}{\mathbf{Z}}

\newcommand{\fb}{\mathfrak{b}}

\newcommand{\fm}{\mathfrak{m}}

\newcommand{\fp}{\mathfrak{p}}

\newcommand{\arxiv}[1]{\href{http://arxiv.org/abs/#1}{{\tiny\tt arXiv:#1}}}
\newcommand{\DOI}[1]{\href{http://doi.org/#1}{\color{purple}{\tiny\tt DOI:#1}}}
\newcommand{\defn}[1]{\emph{#1}}
\let\ol\overline
\let\ul\underline
\renewcommand{\phi}{\varphi}
\renewcommand{\emptyset}{\varnothing}
\DeclareMathOperator{\im}{im} 
\DeclareMathOperator{\End}{End}

\DeclareMathOperator{\Aut}{Aut}

\DeclareMathOperator{\Hom}{Hom}

\DeclareMathOperator{\Spec}{Spec}

\newcommand{\id}{\mathrm{id}}

\renewcommand{\Vec}{\mathrm{Vec}}

\DeclareMathOperator{\Frac}{Frac}

\newcommand{\bbone}{\mathds{1}}

\newcommand{\uotimes}{\mathbin{\ul{\otimes}}}

\DeclareMathOperator{\Amalg}{Amalg}
\DeclareMathOperator{\udim}{\ul{dim}}
\DeclareMathOperator{\utr}{\ul{tr}}

\let\lbb\llbracket
\let\rbb\rrbracket

\contourlength{1pt}

\contourlength{0.8pt}
\newcommand{\myuline}[1]{%
  \uline{\phantom{#1}}%
  \llap{\contour{white}{#1}}%
}
\DeclareMathOperator{\uRep}{\text{\myuline{\rm Rep}}}
\DeclareMathOperator{\uPerm}{\ul{Perm}}

\tikzset{leaf/.style={circle,fill=black,draw,minimum size=1mm,inner sep=0pt}}
\tikzset{boron/.style={circle,fill=white,draw,minimum size=1mm,inner sep=0pt}}
\tikzset{chosen/.style={star, star points=17, star point ratio=20,fill=black,draw,minimum size=2.5mm,inner sep=0pt}}
\tikzset{chosen2/.style={fill=black,draw,minimum size=1.5mm,inner sep=0pt}}
\tikzset{mid/.style={circle,fill=black,draw,minimum size=0.05mm,inner sep=0pt}}

\title{Arboreal tensor categories}

\author{Nate Harman}
\author{Ilia Nekrasov}
\author{Andrew Snowden}
\thanks{AS was supported by NSF grant DMS-2301871.}

\date{March 17, 2024}

\begin{document}

\begin{abstract}
We introduce some new symmetric tensor categories based on the combinatorics of trees: a discrete family $\cD(n)$, for $n \ge 3$ an integer, and a continuous family $\cC(t)$, for $t \ne 1$ a complex number. The construction is based on the general oligomorphic theory of Harman--Snowden, but relies on two non-trivial results we establish. The first determines the measures for the class of trees, and the second is a semi-simplicity theorem. These categories have some notable properties: for instance, $\cC(t)$ is the first example of a 1-parameter family of pre-Tannakian categories of superexponential growth that cannot be obtained by interpolating categories of moderate growth.
\end{abstract}

\maketitle
\tableofcontents

\section{Introduction}

The purpose of this paper is to introduce some new symmetric tensor categories based on the combinatorics of trees. The construction is based on the general framework of \cite{repst}, but depends on two non-trivial results, which constitute the main theorems of this paper. The first is a combinatorial result about trees, namely, the classification of measures. The second is an algebraic result asserting that certain algebras and categories are semi-simple. We discuss these two results in some detail in \S \ref{ss:intro-meas} and \S \ref{ss:intro-ss}, and then discuss the motivation and significance of our work in \S \ref{ss:intro-motiv}.

\subsection{Measures on trees} \label{ss:intro-meas}

Suppose that $T_1$ and $T_2$ are two finite trees, with leaves indexed by the disjoint sets $I_1$ and $I_2$. An \defn{amalgamation} of $T_1$ and $T_2$ is a tree $T$ with leaves labeled by $I_1 \cup I_2$ such that the leaves labeled by $I_{\alpha}$ span a subtree of $T$ homeomorphic to $T_{\alpha}$ for $\alpha=1,2$. Leaves of $T$ can carry two labels (one from $I_1$ and one from $I_2$), meaning the copies of $T_1$ and $T_2$ inside $T$ are allowed to intersect. More generally, if $T_1$ and $T_2$ have a common subtree $T_0$, then an amalgamation of $T_1$ and $T_2$ \defn{over} $T_0$ is an amalgamation where the two copies of $T_0$ are identified. The study of amalgamations of trees originates in Cameron's work \cite{CameronTrees}.

As an example, suppose $T_1$ and $T_2$ are the following trees:
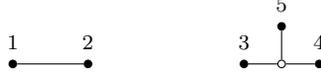
\begin{displaymath}
\tikz{
\node[leaf,label={\tiny 1}] (x) at (-0.5,0) {};
\node[leaf,label={\tiny 2}] (y) at (.5,0) {};
\path[draw] (x)--(y);
}
\qquad\qquad
\tikz{
\node[boron] (A) at (0,0) {};
\node[leaf,label={\tiny 3}] (x) at (-0.5,0) {};
\node[leaf,label={\tiny 4}] (y) at (.5,0) {};
\node[leaf,label={\tiny 5}] (a) at (0,.5) {};
\path[draw] (A)--(x);
\path[draw] (A)--(y);
\path[draw] (A)--(a);
}
\end{displaymath}
The following three trees are examples of amalgamations:
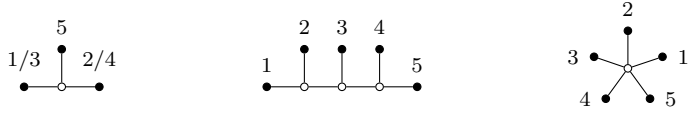
\begin{displaymath}
\tikz[baseline=0pt]{
\node[boron] (A) at (0,0) {};
\node[leaf,label={\tiny 1/3}] (x) at (-0.5,0) {};
\node[leaf,label={\tiny 2/4}] (y) at (.5,0) {};
\node[leaf,label={\tiny 5}] (a) at (0,.5) {};
\path[draw] (A)--(x);
\path[draw] (A)--(y);
\path[draw] (A)--(a);
}
\qquad\qquad
\tikz[baseline=0pt]{
\node[boron] (A) at (0,0) {};
\node[boron] (B) at (.5,0) {};
\node[boron] (C) at (1,0) {};
\node[leaf,label={\tiny 1}] (a1) at (-0.5,0) {};
\node[leaf,label={\tiny 2}] (a2) at (0,.5) {};
\node[leaf,label={\tiny 3}] (b) at (.5,.5) {};
\node[leaf,label={\tiny 4}] (c1) at (1,.5) {};
\node[leaf,label={\tiny 5}] (c2) at (1.5,0) {};
\path[draw] (A)--(B);
\path[draw] (B)--(C);
\path[draw] (A)--(a1);
\path[draw] (A)--(a2);
\path[draw] (B)--(b);
\path[draw] (C)--(c1);
\path[draw] (C)--(c2);
}
\qquad\qquad
\tikz[baseline=-7pt]{
\node[boron] (X) at (0,0) {};
\node[leaf,label=right:{\tiny 1}] (A) at (.46,.154) {};
\node[leaf,label={\tiny 2}] (B) at (0,.5) {};
\node[leaf,label=left:{\tiny 3}] (C) at (-.45,.154) {};
\node[leaf,label=left:{\tiny 4}] (D) at (-.294,-.404) {};
\node[leaf,label=right:{\tiny 5}] (E) at (.294,-.404) {};
\draw (X)--(A);
\draw (X)--(B);
\draw (X)--(C);
\draw (X)--(D);
\draw (X)--(E);
}
\end{displaymath}
There are 56 amalgamations in total. If $T_0$ is the one-point tree, identified with vertex~1 in $T_1$ and vertex~3 in $T_2$, then the first amalgamation above is over $T_0$, while the other two are not. There are six amalgamations over $T_0$.

A \defn{regular measure} on the class $\fT$ of trees, valued in a ring $k$, is a rule $\rho$ that assigns to each (isomorphism class of) tree $T$ a unit $\rho(T)$ of $k$ such that the empty tree is assigned~1, and the following \defn{amalgamation equation} holds: given $T_1$ and $T_2$ with a common subtree $T_0$, we have
\begin{displaymath}
\frac{\rho(T_1) \rho(T_2)}{\rho(T_0)} = \sum_{i=1}^n \rho(U_i),
\end{displaymath}
where $U_1, \ldots, U_n$ are the amalgamations of $T_1$ and $T_2$ over $T_0$. For example, if $T_1$ and $T_2$ are as in the previous paragraph and $T_0$ is the empty subtree, then the amalgamation equation takes the form
\begin{align*}
\rho\big(\,
\tikz[baseline=-2pt]{
\node[leaf] (x) at (-0.25,0) {};
\node[leaf] (y) at (.25,0) {};
\path[draw] (x)--(y);
}
\,\big) \cdot \rho\big(\,
\tikz[baseline=-3pt]{
\node[boron] (A) at (0,0) {};
\node[leaf] (x) at (-0.433,-0.25) {};
\node[leaf] (y) at (0.433,-0.25) {};
\node[leaf] (a) at (0,.5) {};
\path[draw] (A)--(x);
\path[draw] (A)--(y);
\path[draw] (A)--(a);
}
\,\big) =&
\rho\big(\,
\tikz[baseline=-3pt]{
\node[boron] (X) at (0,0) {};
\node[leaf] (A) at (.46,.154) {};
\node[leaf] (B) at (0,.5) {};
\node[leaf] (C) at (-.45,.154) {};
\node[leaf] (D) at (-.294,-.404) {};
\node[leaf] (E) at (.294,-.404) {};
\draw (X)--(A);
\draw (X)--(B);
\draw (X)--(C);
\draw (X)--(D);
\draw (X)--(E);
}
\,\big) + 15 \rho\big(\,
\tikz[baseline=-3pt]{
\node[boron] (A) at (0,0) {};
\node[boron] (B) at (.5,0) {};
\node[boron] (C) at (1,0) {};
\node[leaf] (a1) at (-0.35,0.35) {};
\node[leaf] (a2) at (-0.35,-0.35) {};
\node[leaf] (b) at (.5,.5) {};
\node[leaf] (c1) at (1.35,0.35) {};
\node[leaf] (c2) at (1.35,-0.35) {};
\path[draw] (A)--(B);
\path[draw] (B)--(C);
\path[draw] (A)--(a1);
\path[draw] (A)--(a2);
\path[draw] (B)--(b);
\path[draw] (C)--(c1);
\path[draw] (C)--(c2);
}
\,\big) + 10 \rho\big(\,
\tikz[baseline=-3pt]{
\node[boron] (A) at (0,0) {};
\node[boron] (C) at (0.5,0) {};
\node[leaf] (a1) at (-0.35,0.35) {};
\node[leaf] (a2) at (-0.35,-0.35) {};
\node[leaf] (c1) at (0.85,0.35) {};
\node[leaf] (c2) at (1,0) {};
\node[leaf] (c3) at (0.85,-0.35) {};
\path[draw] (A)--(C);
\path[draw] (A)--(a1);
\path[draw] (A)--(a2);
\path[draw] (C)--(c1);
\path[draw] (C)--(c2);
\path[draw] (C)--(c3);
} \,\big) \\
&+ 6 \rho\big(\,
\tikz[baseline=-3pt]{
\node[boron] (A) at (0,0) {};
\node[leaf] (x) at (-.5,0) {};
\node[leaf] (y) at (.5,0) {};
\node[leaf] (a) at (0,.5) {};
\node[leaf] (b) at (0,-.5) {};
\path[draw] (A)--(x);
\path[draw] (A)--(y);
\path[draw] (A)--(a);
\path[draw] (A)--(b);
}
\,\big) + 18 \rho\big(\,
\tikz[baseline=-3pt]{
\node[boron] (A) at (0,0) {};
\node[boron] (C) at (.5,0) {};
\node[leaf] (a1) at (-0.35,0.35) {};
\node[leaf] (a2) at (-0.35,-0.35) {};
\node[leaf] (c1) at (0.85,0.35) {};
\node[leaf] (c2) at (0.85,-0.35) {};
\path[draw] (A)--(C);
\path[draw] (A)--(a1);
\path[draw] (A)--(a2);
\path[draw] (C)--(c1);
\path[draw] (C)--(c2);
}
\,\big) + 6 \rho\big(\,
\tikz[baseline=-3pt]{
\node[boron] (A) at (0,0) {};
\node[leaf] (x) at (-0.433,-0.25) {};
\node[leaf] (y) at (0.433,-0.25) {};
\node[leaf] (a) at (0,.5) {};
\path[draw] (A)--(x);
\path[draw] (A)--(y);
\path[draw] (A)--(a);
}
\,\big).
\end{align*}
The coefficients represent the number of amalgamations with the given isomorphism type. More generally, a \defn{measure} on $\fT$ assigns to each embedding $T' \subset T$ of trees a quantity in $k$ such that some conditions hold (see Definition~ \ref{defn:meas}); a regular measure $\rho$ yields a measure by assigning to $T' \subset T$ the quantity $\rho(T)/\rho(T')$.

Determining the (regular) measures for $\fT$ is an important problem, since it is exactly the input required to build a tensor category (see \S \ref{ss:intro-ss}). It is also challenging: one essentially has a variable for each tree, and a non-linear equation for each amalgamation situation; moreover, we have seen that, even in simple situations, these equations are combinatorially quite complicated. Our first result is a complete solution of this problem. To state the answer, we first note that there is a universal measure valued in a ring $\Theta(\fT)$, in the sense that giving a measure valued in $k$ is equivalent to giving a ring homomorphism $\Theta(\fT) \to k$. Thus to determine all measures, it suffices to compute $\Theta(\fT)$. This is what we do:

\begin{theorem} \label{mainthm}
We have an isomorphism $\Theta(\fT) \cong \bZ[u,v]/(uv)$.
\end{theorem}

See Theorem~\ref{thm:Theta} for additional details. The theorem implies that there are two 1-parameter families of measures. We let $\mu_t$ be the measure obtained by putting $v=0$ and $u=-(t-2)/(t-1)$. This measure is regular if $k$ is a field of characteristic~0 and $t \in k \setminus \bN$. We give an explicit formula for $\mu_t$ in Proposition~\ref{prop:mu-formula}; for example:
\begin{displaymath}
\mu_t \big( \, \tikz[baseline=3pt]{
\node[boron] (A) at (0,0) {};
\node[leaf] (x) at (-0.5,0) {};
\node[leaf] (y) at (.5,0) {};
\node[leaf] (a) at (0,.5) {};
\path[draw] (A)--(x);
\path[draw] (A)--(y);
\path[draw] (A)--(a);
}
\,\big) = -\frac{t(t-2)}{(t-1)^3}, \qquad
\mu_t \big( \, \tikz[baseline=3pt]{
\node[boron] (A) at (0,0) {};
\node[boron] (C) at (.5,0) {};
\node[leaf] (a1) at (-0.5,0) {};
\node[leaf] (a2) at (0,.5) {};
\node[leaf] (c1) at (.5,.5) {};
\node[leaf] (c2) at (1,0) {};
\path[draw] (A)--(C);
\path[draw] (A)--(a1);
\path[draw] (A)--(a2);
\path[draw] (C)--(c1);
\path[draw] (C)--(c2);
}
\,\big) = \frac{t(t-2)^2}{(t-1)^4}.
\end{displaymath}
The measures with $u=0$ are never regular, and we will not say much about them.

To prove this theorem, we apply a method due to the second author \cite{Nekrasov} that gives a general approach to computing $\Theta$. This method, combined with some combinatorial properties of trees that we establish, yields a finite presentation of the ring $\Theta(\fT)$. We include a complete exposition of the general method in \S \ref{s:fraisse}.

There is one more result about measures to mention. Let $\fT_n$ be the class of trees in which each vertex has valence at most $n$. One can then define (regular) measures on this class; in this case, one only considers amalgamations that belong to $\fT_n$. We show that the measure $\mu_t$ with $t=n$ induces a regular measure on $\fT_n$. In fact, it will be shown in \cite{Nekrasov} that $\fT_n$ has only two measures (the other one comes from the $u=0$ family).

\subsection{Arboreal algebras and categories} \label{ss:intro-ss}

Let $k$ be a field and fix $t \in k\setminus \{1\}$. Given a tree $T$, we define a $k$-algebra $A^T_k(t)$ as follows. The elements are formal linear combinations of self-amalgamations of $T$. Given self-amalgamations $X$ and $Y$, their product is defined by
\begin{displaymath}
[X] \cdot [Y] = \sum_Z c_{X,Y}^Z [Z], \qquad c_{X,Y}^Z = \sum_A \mu_t(Z \subset A),
\end{displaymath}
where $Z$ varies over all self-amalgamations of $T$, and $A$ varies of all triple self-amalgamations of $T$ such that $A_{12}=X$, $A_{23}=Y$, and $A_{13}=Z$, where $A_{ij}$ denotes the subtree spanned by the indices of type $i$ and $j$. See Example~\ref{ex:arboreal} for a more detailed explanation of multiplication. Properties of the measure $\mu_t$ ensure that the multiplication is associative and unital. We call $A^T_k(t)$ an \defn{arboreal algebra}. It is similar in some respects to other diagram algebras, like the partition algebras, but, as far as we know, it has not been previously considered.

More generally, we define a category based on trees, as follows. The objects are formal direct sums of trees; we write $\Vec_X$ for the object corresponding to the tree $X$. Morphisms $\Vec_X \to \Vec_Y$ are formal linear combinations of amalgamations of $X$ and $Y$. Composition is defined analogously to multiplication in the arboreal algebra; in particular, $A^T_k(t)$ is the endomorphism algebra of $\Vec_T$ in this category. This category carries a tensor product, defined on objects by $\Vec_X \otimes \Vec_Y=\bigoplus_Z \Vec_Z$, where the sum is over all amalgamations.

We define $\cC_k(t)$ to be the Karoubi envelope of the above category. There is also a variant $\cD_k(n)$, for $n \ge 3$ an integer, where we use the class $\fT_n$ of trees of bounded valence and the measure induced by $\mu_n$. We refer to these categories $\cC_k(t)$ and $\cD_k(n)$ as \defn{arboreal tensor categories}. Our second theorem establishes some basic properties of these categories.

\begin{theorem} \label{mainthm2}
Let $k$ be a field, let $n \ge 3$ be an integer, and let $t \in k \setminus \{1\}$.
\begin{enumerate}
\item If $n!$ is invertible in $k$ then $\cD_k(n)$ is a semi-simple pre-Tannakian category.
\item If $\operatorname{char}(k)=0$ and $t \not\in \bN$ then $\cC_k(t)$ is a semi-simple pre-Tannakian category.
\item Suppose $\operatorname{char}(k)=0$ and $t=n$. Then there is a tensor functor $\cC_k(t) \to \cD_k(n)$ that realizes the target as the semi-simplification of the source.
\end{enumerate}
\end{theorem}

See \S \ref{s:arboreal} for more detailed statements, and \S \ref{ss:tensor} for terminology related to tensor categories. The construction of the arboreal tensor categories comes from the general machinery of \cite{repst}. The proof of the above theorem draws on some ideas from \cite{repst} as well, but also involves a few new ingredients. To prove~(a), we introduce a new semisimplicity criterion in \S \ref{ss:ss-crit}. To prove~(b), we identify certain specializations of the arboreal algebra $A^T_k(t)$ in positive characteristic with endomorphism algebras in $\cD_k(n)$, and thereby deduce that they are semi-simple. We obtain (c) as a special case of a more general result.

%This theorem is significant for two reasons. First, new pre-Tannakian categories are rather hard to come by, and so the theorem shows that we have found some rather special objects. And second, part (c) shows that the categories $\cC_k(t)$ can be viewed as interpolating the categories $\cD_k(n)$ as $n$ varies, in the sense of Deligne (see also Remark~\ref{rmk:ultra}). We elaborate on both of these points below.

\subsection{Motivation} \label{ss:intro-motiv}

Pre-Tannakian categories are a natural class of tensor categories generalizing representation categories of algebraic (super)groups. Deligne \cite{Deligne2} characterized classical representation categories as those pre-Tannakian categories of moderate (exponential) growth, in characteristic~0. Constructing pre-Tannakian categories of superexponential growth is a difficult and important problem.

Deligne \cite{Deligne2} explained how one can sometimes ``interpolate'' a sequence of pre-Tannakian categories to obtain a new such category. By interpolating the representation categories of finite symmetric groups, he constructed a category $\uRep(\fS_t)$, where $t$ is a complex number, that has superexponential growth. Shortly thereafter, Knop \cite{Knop, Knop2} generalized this procedure and gave a number of new examples of interpolation categories.

Recently, two of us \cite{repst} gave a new construction of pre-Tannakian categories based on oligomorphic groups, or, equivalently, Fra\"iss\'e classes. In that paper, we constructed (essentially) two new examples of pre-Tannakian categories; these were the first examples that cannot be obtained by the interpolation method. Since then a small number of additional examples have been found \cite{Kriz, dblexp, cantor}.

In this paper, we construct two infinite families of new pre-Tannakian categories: the discrete family $\cD_k(n)$ and the continuous family $\cC_k(t)$. Producing such a large number of new pre-Tannakian categories is interesting enough on its own to justify this paper. However, there two additional reasons that these categories are particularly interesting.

First, while we do not officially define the category $\cD_k(2)$, it really should be the Delannoy category studied in \cite{line}; see \S \ref{ss:further}(b). The results of \cite{line} show that this category is extremely special. For instance, all of its simple objects have categorical dimension $\pm 1$ and the Adams operations on its Grothendieck group are trivial. Since the categories $\cD_k(n)$ naturally extend the Delannoy category, it is quite likely that they too are special objects.

Second, the family $\cC_k(t)$ is the first example of a continuous family of pre-Tannakian categories that cannot be obtained by interpolating classical representation categories; all previously found examples isolated, in the sense that they do not deform. Moreover, these categories \emph{can} be obtained by applying Deligne's interpolation procedure to the discrete family $\cD_k(n)$. This is the first non-trivial example of Deligne interpolation where the initial categories are already of superexponential growth.

\subsection{Open problems}

We highlight three problems arising from this work.

(a) The approach to $\Theta(\fT)$ given in this paper provides a conceptual explanation for why one should be able to compute this ring: the abundance of separated vertices in trees means the computation greatly degenerates. However, this does not give a conceptual explanation for why $\Theta(\fT)$ is non-zero; indeed, we are only able to show this after going through all the details of the computation. Is there a less computational way to see that $\Theta(\fT)$ is non-zero in this case? We note that the first two authors have found an example of a Fra\"iss\'e class where $\Theta$ vanishes, namely, the variant of $\fT_3$ where one considers planar trees.

(b) Describe the structure of the categories $\cD_k(n)$ and $\cC_k(t)$. For instance, we would like to have a parametrization of the simple objects, a formula for their categorical dimensions, a combinatorial formula for tensor products, and so on. Knop \cite{Knop3} solved these problems for a general class of categories (which does not include the arboreal categories), and in the case of the Delannoy category $\cD_k(2)$ very detailed answers are provided in \cite{line}.

(c) Construct an abelian envelope of $\cC_k(t)$ when $t \in \bN$ is a singular parameter. There are some parallels between the categories $\cC_k(t)$ and Deligne's interpolation categories $\uRep(\fS_t)$, and such abelian envelopes were constructed in that setting by Comes and Ostrik \cite{ComesOstrik}. We would also like to have an abelian envelope for the tensor categories corresponding to the second family of measures on $\fT$ (the $u=0$ family). These are special cases of a general problem, which seems to be very important: namely, to construct abelian envelopes for the categories $\uPerm(\fF; \mu)$ when $\mu$ is not a quasi-regular measure. The quasi-regular case is reasonably well-understood by \cite{repst}, but outside of this case there are almost no results.

\subsection{Outline}

In \S \ref{s:fraisse} we discuss measures on general Fra\"iss\'e classes. We specialize to the case of trees in \S \ref{s:trees}. In \S \ref{s:tencat}, we discuss the tensor categories associated to general Fra\"iss\'e classes. Again, we specialize to the case of trees in \S \ref{s:arboreal}, where we prove the most important results in this paper. Finally, in \S \ref{s:example}, we work out some examples in arboreal tensor categories.

\subsection{Notation}

We list some of the most important notation here:
\begin{description}[align=right,labelwidth=2.25cm,leftmargin=!]
\item[ $k$ ] the coefficient ring
\item[ $\fF$ ] a general Fra\"iss\'e class (\S \ref{ss:fraisse})
\item[ $\fT$ ] the Fra\"iss\'e class of trees (\S \ref{ss:trees})
\item[ $\fT_n$ ] the Fra\"iss\'e class of trees of level $\le n$ (\S \ref{ss:trees})
\item[ $\mu_t$ ] the measure for $\fT$ with parameter $t$ (\S \ref{ss:goodmeas})
\item[ $\cC_k(t)$ ] the arboreal category with parameter $t$ (\S \ref{ss:treecat})
\item[ $\cD_k(n)$ ] the arboreal category of level $\le n$ (\S \ref{ss:treecat})
\end{description}

\section{Measures on Fra\"iss\'e classes} \label{s:fraisse}

\subsection{Overview}

In \cite{repst}, we constructed a tensor category $\uPerm_k(\fF; \mu)$ associated to a Fra\"iss\'e class $\fF$ and a measure $\mu$ for $\fF$ valued in the commutative ring $k$. The main difficulty in applying this construction, typically, is constructing measures for $\fF$. In \cite{repst}, we defined a universal measure valued in a certain ring $\Theta(\fF)$. If one can compute this ring, then one will understand all measures for $\fF$. The main result of this section, Theorem~\ref{thm:refined}, gives a presentation for $\Theta(\fF)$ that, in certain cases, is practically computable. We emphasize that this theorem is due to the second author, and will appear in his thesis \cite{Nekrasov}.

The main point of this paper is to construct the arboreal tensor categories. These will be derived from the general construction $\uPerm_k(\fF; \mu)$ by taking $\fF$ to be a Fra\"iss\'e class of trees. We will apply Theorem~\ref{thm:refined} to compute $\Theta(\fF)$ for these Fra\"iss\'e classes.

We briefly outline the contents of this section. In \S \ref{ss:fraisse} we recall the definition of Fra\"iss\'e class, and in \S \ref{ss:meas} we recall the definition of measure. In \S \ref{ss:Theta-present} we give an initial presentation for $\Theta(\fF)$. We refine this presentation in \S \ref{ss:separated} and \S \ref{ss:refined} to obtain the more useful presentation of the second author. We close in \S \ref{ss:induced} with one additional result about measures.

\subsection{Fra\"iss\'e classes} \label{ss:fraisse}

A \defn{signature} is a pair $\Sigma=(I,n)$ where $I$ is an index set and $n \colon I \to \bZ_{\ge 1}$ is a function. Fix a signature $\Sigma$. A \defn{(relational) structure} is a set $X$ together with for each $i \in I$ a subset $R_i$ of $X^{n(i)}$. We think of $R_i$ as a relation on $X$ of arity $n(i)$, i.e., $R_i$ takes as input elements $x_1, \ldots, x_{n(i)}$ in $X$ and returns true or false. For example, if $I=\{\ast\}$ is a singleton with $n(\ast)=2$ then a structure is a set with a single binary relation.

There is an obvious notion of isomorphism of structures. If $Y$ is a structure and $X$ is a subset of $Y$ then there is an induced structure on $X$: simply restrict each relation $R_i$ to $X$. An \defn{embedding} of structures $i \colon X \to Y$ is an injection of sets such that $i$ induces an isomorphism between $X$ and $i(X)$ equipped with the induced structure.

Suppose that $i_1 \colon X \to Y_1$ and $i_2 \colon X \to Y_2$ are embeddings of structures. An \defn{amalgamation} of $Y_1$ and $Y_2$ over $X$ is a triple $(Z, j_1, j_2)$ where $Z$ is a structure and $j_1 \colon Y_1 \to Z$ and $j_2 \colon Y_2 \to Z$ are embeddings such that $j_1 \circ i_1=j_2 \circ i_2$ and $Z=\im(j_1) \cup \im(j_2)$. An \defn{isomorphism} of amalgamations $(Z, j_1, j_2) \to (Z', j_1', j_2')$ is an isomorphism of structures $k \colon Z \to Z'$ such that $k \circ j_1=j_1'$ and $k \circ j_2=j_2'$. An amalgamation has no non-trivial automorphisms; in particular, if two amalgamations are isomorphic then they are uniquely isomorphic. We say that a class $\fF$ of structures (for a fixed signature) has the \defn{amalgamation property} if whenever $i_1 \colon X \to Y_1$ and $i_2 \colon X \to Y_2$ are embeddings in $\fF$, there is at least one amalgamation $(Z, j_1, j_2)$ with $Z$ in $\fF$.

\begin{definition}
A \defn{Fra\"iss\'e class} is a non-empty class $\fF$ of finite structures for a fixed signature satisfying the following properties:
\begin{enumerate}
\item For each $n \ge 0$, there are only finitely many isomorphism classes of $n$-element structures in $\fF$.
\item The class $\fF$ is \defn{hereditary}: if $X \in \fF$ and $Y$ embeds into $X$ then $Y \in \fF$.
\item The class $\fF$ satisfies the amalgamation property. \qedhere
\end{enumerate}
\end{definition}

Condition (a) is often relaxed to ``$\fF$ contains countably many isomorphism classes,'' but the stronger version we use is necessary when working with measures. If the signature is finite (meaning $I$ is a finite set) then (a) is automatic.

Fra\"iss\'e classes play an important role in certain parts of model theory. We refer to \cite{CameronBook, Macpherson} for general background in this direction.

\begin{example}
The class of all finite sets is a Fra\"iss\'e class; here the signature is empty.
\end{example}

\begin{example}
The class of all finite totally ordered sets is a Fra\"iss\'e class; here the signature consists of a single binary relation determining the total order.
\end{example}

\subsection{Measures} \label{ss:meas}

We now recall the following crucial notion, introduced in \cite[Definition 6.2]{repst}:

\begin{definition} \label{defn:meas}
Let $\fF$ be a Fra\"iss\'e class. A \defn{measure} for $\fF$ with values in a commutative ring $k$ is a rule $\mu$ that assigns to each embedding $i \colon X \to Y$ of structures in $\fF$ a value $\mu(i)$ in $k$ such that the following conditions hold:
\begin{enumerate}
\item If $i$ is an isomorphism then $\mu(i)=1$.
\item If $i \colon X \to Y$ and $j \colon Y \to Z$ are composable embeddings then $\mu(j \circ i)=\mu(j) \cdot \mu(i)$.
\item Let $i \colon X \to Y$ and $X \to X'$ be embeddings, and let $i'_{\alpha} \colon X' \to Y'_{\alpha}$, for $1 \le \alpha \le n$, be the various amalgamations. Then $\mu(i)=\sum_{\alpha=1}^n \mu(i'_{\alpha})$.
\end{enumerate}
A measure $\mu$ is \defn{regular} if $\mu(i)$ is a unit of $k$ for all $i$.
\end{definition}

Suppose that $\mu$ is a measure for $\fF$. Given a structure $X$ in $\fF$, we define $\mu(X)$ to be the value of $\mu$ on the embedding $\emptyset \to X$. If $i \colon X \to Y$ is an embedding, then by considering the composition $\emptyset \to X \to Y$, we see that $\mu(X) \cdot \mu(i) = \mu(Y)$. Thus if $\mu(X)$ is a unit of the ring $k$ then $\mu(i)=\mu(Y) \cdot \mu(X)^{-1}$. In particular, a regular measure is determined by its values on structures.

An important, and often difficult, problem is to determine the measures on a particular Fra\"iss\'e class $\fF$. The following is a useful object to consider in the context of this problem:

\begin{definition} \label{defn:theta}
Let $\fF$ be a Fra\"iss\'e class. We define a commutative ring $\Theta(\fF)$ as follows. For each embedding $i \colon X \to Y$ in $\fF$ there is a class $[i]$ in $\Theta(\fF)$. These classes satisfy relations reflecting the axioms of measure. Explicitly:
\begin{enumerate}
\item If $i$ is an isomorphism then $[i]=1$.
\item If $i$ and $j$ are composable then $[j \circ i] = [j] \cdot [i]$.
\item With notation as in Definition~\ref{defn:meas}(c), $[i]=\sum_{\alpha=1}^n [i'_{\alpha}]$.
\end{enumerate}
A bit more precisely, $\Theta(\fF)=R/I$, where $R$ is the polynomial ring on symbols $[i]$, and $I$ is the ideal generated by the relations in (a)--(c) above.
\end{definition}

There is a measure $\mu_{\rm univ}$ for $\fF$ valued in the ring $\Theta(\fF)$ defined by $\mu_{\rm univ}(i)=[i]$. This measure is universal in the following sense: if $\mu$ is an arbitrary measure for $\fF$ valued in a commutative ring $k$ then there is a unique ring homomorphism $\phi \colon \Theta(\fF) \to k$ such that $\mu = \phi \circ \mu_{\rm univ}$. In particular, to classify all measures for $\fF$, it is enough to compute the ring $\Theta(\fF)$.

\begin{remark}
In this paper, we are mostly focused on measures for a Fra\"iss\'e class $\fF$ and the associated ring $\Theta(\fF)$. In \cite{repst}, we instead focused on measures for an oligomorphic group $G$ and the associated ring $\Theta(G)$. If $G$ arises from $\fF$ then these two notions are closely related, but need not coincide exactly; see \S \ref{ss:oligo} for details.
\end{remark}

\subsection{A presentation for $\Theta$} \label{ss:Theta-present}

Fix a Fra\"iss\'e class $\fF$. We say that an embedding $Y \to X$ is a \defn{one-point extension} if $\# X = 1+ \# Y$. Every embedding can be factored into a sequence of one-point extensions, and so Definition~\ref{defn:theta}(b) implies that $\Theta(\fF)$ is generated by the corresponding classes of one-point extensions. We now give a presentation in terms of these generators.

We first make a slight shift in perspective. Specifying a one-point extension $Y \to X$ is the same as specifying a point of $X$, namely, the unique point not in $Y$. A \defn{marked structure} is a pair $(X,x)$ consisting of a structure $X$ in $\fF$ and a point $x \in X$. Given $(X,x)$, we have the associated one-point extension $i \colon X \setminus x \to X$, and we write $[X,x]$ for the class $[i]$ in $\Theta(\fF)$. We thus see that $\Theta(\fF)$ is generated by the classes of marked structures.

Let $P$ be the polynomial ring (over $\bZ$) with variables indexed by (isomorphism classes of) marked structures. We write $\lbb X,a \rbb$ for the variable corresponding to the marked structure $(X,a)$. We have a ring homomorphism $\phi \colon P \to \Theta(\fF)$ satisfying $\lbb X,a \rbb \mapsto [X,a]$, which is surjective by the previous paragraph. We now write down some elements in the kernel.

An \defn{L-datum} is a triple $((X,a), (Y,b), i)$ where $(X,a)$ and $(Y,b)$ are marked structures and $i \colon X \setminus a \to Y \setminus b$ is an isomorphism. We sometimes simply write $(X,Y)$ for brevity. Fix such a datum. Put $\delta=1$ if $i$ extends to an isomorphism $X \to Y$, and $\delta=0$ otherwise. Let $X_1, \ldots, X_n$ be the proper amalgamations of $X$ and $Y$ over $X \setminus a \cong Y \setminus b$, i.e., those in which $a$ and $b$ are not identified. We define $L((X,a), (Y,b), i)$ to be the element
\begin{displaymath}
\lbb X,a \rbb - \delta - \sum_{i=1}^n \lbb X_i, a \rbb
\end{displaymath}
of $P$. This belongs to the kernel of $\phi$ by Definition~\ref{defn:theta}(a). We let $\fb_1$ be the ideal of $P$ generated by these elements, as we vary over all L-data. If $(X,a)$ is a marked structure then we obtain an L-datum $((X,a),(X,a),\id)$, where $\id$ is the identity map on $X \setminus a$. We refer to this datum as the \defn{duplicate} of $(X,a)$, and write $L(X,a)$ for the corresponding element of $\fb_1$.

A \defn{Q-datum} is a triple $(X,a,b)$ where $X$ is a structure in $\fF$ and $a,b \in X$ are distinct elements. Given such a datum, we define $Q(X,a,b)$ to be the element
\begin{displaymath}
\lbb X, a \rbb \cdot \lbb X \setminus a, b \rbb - \lbb X, b \rbb \cdot \lbb X \setminus b, a \rbb
\end{displaymath}
of $P$. This belongs to the kernel of $\phi$ by Definition~\ref{defn:theta}(b). We define $\fb_2$ to be the ideal of $P$ generated by these elements, as we carry over all Q-data. We say that the Q-datum $(X,a,b)$ is \defn{symmetric} if there is an automorphism of $X$ switching $a$ and $b$. In this case, the element $Q(X,a,b)$ vanishes.

We put $\fb=\fb_1+\fb_2$. The following is the main result we are after.

\begin{proposition} \label{prop:Theta-present}
The map $\phi$ induces an isomorphism $P/\fb \to \Theta(\fF)$.
\end{proposition}

\begin{proof}
We define a homomorphism $\psi \colon \Theta(\fF) \to P/\fb$, by specifying a measure valued in $P/\fb$. Let $i \colon Y \to X$ be an embedding. Write $X \setminus Y = \{x_1, \ldots, x_n\}$, and let $X_a=Y \cup \{x_1, \ldots, x_a\}$ for $1 \le a \le n$. We define $\psi(i)$ to be the class of $\lbb X_1, x_1 \rbb \cdots \lbb X_n, x_n \rbb$ in $P/\fb$. This is independent of the enumeration of $X \setminus Y$ by the relations in $\fb_2$.

It is clear that $\psi$ satisfies conditions (a) and (b) from Definition~\ref{defn:meas}. We now verify (c). We are given that (c) holds when both embeddings are one-point extensions, by the relations in $\fb_1$. One can show that if (c) holds for two morphisms then it holds for their composition. (This is easier to see on the oligomorphic group side, using the fact that the fiber product of a composition is the composition of fiber products; see \S \ref{ss:oligo} and \cite[\S 4.5]{repst}.) Since every morphism is a composition of one-point extensions, (c) holds in general.

We have thus shown that $\psi$ is well-defined. It is clear that $\phi$ and $\psi$ are inverses, so the result follows.
\end{proof}

\begin{example}
Let $\fF$ be the Fra\"iss\'e class of finite sets. For each $n \ge 1$, there is a unique $n$-element marked structure up to isomorphism; call it $X_n$, and let $x_n$ be the corresponding variable of $P$. Every L-datum is isomorphic to a duplicate of some $X_n$. We have $L(X_n)=x_n-1-x_{n+1}$, and so $\fb_1$ is generated by these forms. Since all Q-data are symmetric, we have $\fb_2=0$. We thus find $\Theta \cong \bZ[x_1]$.
\end{example}

\subsection{Separated elements} \label{ss:separated}

Let $X$ be a structure in $\fF$, and let $a$ and $b$ be distinct elements of $X$. We say that $a$ and $b$ are \defn{separated} if $X$ is the unique amalgamation (in $\fF$) of $X \setminus a$ and $X \setminus b$ over $X \setminus \{a,b\}$. Note that the notion of separated depends on the ambient Fra\"iss\'e class $\fF$. The following is a key observation, due to \cite{Nekrasov}:

\begin{proposition}\label{prop:separated}
If $a$ and $b$ are separated elements of $X$ then $[X,a]=[X \setminus b, a]$ in $\Theta(\fF)$.
\end{proposition}

\begin{proof}
Consider the diagram
\begin{displaymath}
\xymatrix{
X \setminus b \ar[r] & X \\
X \setminus \{a,b\} \ar[u]^i \ar[r] & X \setminus a \ar[u]_j }
\end{displaymath}
Since $X$ is the unique amalgamation, it follows from Definition~\ref{defn:theta}(c) that $[i]=[j]$ in $\Theta(\fF)$, which is exactly the stated relation.
\end{proof}

Let $(X,a)$ be a marked structure. We say that $b \ne a$ is \defn{extraneous} if it is separated from $a$, and we say that $(X,a)$ is \defn{minimal} if there are no extraneous elements. The proposition shows that the classes of minimal marked structures generate $\Theta(\fF)$. Depending on $\fF$, separated elements may be plentiful or scarce, as the following examples illustrate. In the main example of this paper (trees), there will be many separated elements.

\begin{example}
Let $\fF$ be the Fra\"iss\'e class of finite sets. In this case, there are no separated elements. Indeed, if $a,b \in X$ are distinct then there are two amalgamations of $X \setminus a$ and $X \setminus b$ over $X \setminus \{a,b\}$, namely, $X$ and the quotient of $X$ where $a$ and $b$ are identified.
\end{example}

\begin{example} \label{example:DLO}
Let $\fF$ be the Fra\"iss\'e class of finite totally ordered sets. Let $a,b \in X$ be distinct. One can show that $a$ and $b$ are separated if and only if there is some element $c \in X$ strictly in between $a$ and $b$, i.e., $a<c<b$ or $b<c<a$. We thus find that there are four minimal marked structures:
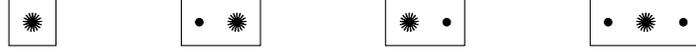
\begin{displaymath}
\tikz[framed]{
\node[chosen] (B) at (0,0) {};
}
\qquad\qquad
\tikz[framed]{
\node[leaf] (B) at (0,0) {};
\node[chosen] (B) at (.5,0) {};
}
\qquad\qquad
\tikz[framed]{
\node[chosen] (B) at (0,0) {};
\node[leaf] (B) at (.5,0) {};
}
\qquad\qquad
\tikz[framed]{
\node[leaf] (B) at (0,0) {};
\node[chosen] (B) at (.5,0) {};
\node[leaf] (B) at (1,0) {};
}
\end{displaymath}
The above sets are ordered from left to right, and the asterisk is the marked element. We name these classes $O_1, O_{r}, O_{l}$, and $O_{m}$ respectively.
%\Icom{, where \textit{r},\textit{l}, and \textit{m} stand for \textit{right}, \textit{left}, and \textit{middle} respectively. These names have the meaning: the variable for a one-point extension $(X, x)$ where $x$ is a right (left) most point in the totally ordered $X$ of an \textit{arbitrary} cardinality at least 2 is equal to $O_{r}$ ($O_l$); if $x$ lies properly in the middle of $X$, with at least one point of $X$ from both sides of $x$, then the variable is forced to be equal to $O_{m}$.}
We see that $\Theta(\fF)$ is generated by $O_1, O_{r} ,O_{l}, O_{m}$.
\end{example}

%\Icom{Ilia is truly afraid of adding more text... However, we can add here something enticing for model-theorists or for audience ready for some model theory.}
%
%\Icom{Like:}
%\begin{remark}
%We do not try to retranslate the method into the language of model theory. However, let us just demonstrate effective naturality of the idea of separated elements and Proposition~\ref{prop:separated} for total orders.
%
%Consider the marked structure $(X = \{x < y < z\}, x)$. Then to specify $x$ inside $X$, by default, we would consider a set of formulas $x < y$ and $x < z$. However, the former implies the latter as we have a formula $y<z$ in $X \backslash x$ and the transitivity axiom. Therefore the class of this marked structure coincides with the class of $(\{ x<y \},x)$ which is $O_l$ in our notation.
%\end{remark}

\subsection{A refined presentation} \label{ss:refined}

We have just seen that $\Theta(\fF)$ is generated by minimal marked structures. We now find a set of ``minimal'' relations between these classes to obtain a presentation.

We first define an equivalence relation $\sim$ on marked structures. Concisely, $\sim$ is the smallest equivalence relation in which isomorphic marked structures are equivalent, and $(X,a) \sim (X \setminus b, a)$ if $b$ is extraneous. In more detail, we have $(X,a) \sim (X',a')$ if there is a chain
\begin{displaymath}
(X,a)=(X_1,a_1), \ldots, (X_n,a_n)=(X',a')
\end{displaymath}
such that $(X_{i+1},a_{i+1})$ is either isomorphic to $(X_i,a_i)$, or obtained from it by adding or deleting an extraneous element.

Let $P^*$ be the polynomial ring on isomorphism classes of minimal marked structures. We still write $\lbb X,a \rbb$ for the variable corresponding to $(X,a)$. We let $\fb_0^*$ be the ideal of $P^*$ generated by the elements $\lbb X,a \rbb-\lbb X',a'\rbb$ where $(X,a)$ and $(X',a')$ are equivalent minimal marked structures. There is a natural ring homomorphism $\pi \colon P \to P^*/\fb_0^*$ defined by mapping $\lbb X,a \rbb$ to the (class of) $\lbb X', a \rbb$, where $X'$ is obtained from $X$ by deleting extraneous elements, one at a time, until a minimal structure is reached. We note that $(X',a)$ is not necessarily well-defined, but its class is well-defined modulo $\fb_0^*$, which is why $\pi$ is a well-defined map.

We now study the images of the ideals $\fb_1$ and $\fb_2$ under $\pi$. We say that distinct elements $a$ and $b$ of a structure $X$ in $\fF$ are \defn{stably separated} if for every embedding $i \colon X \to X'$ in $\fF$, the images $i(a)$ and $i(b)$ are separated elements of $X'$.

\begin{proposition} \label{prop:Lmin}
Let $(X,X')$ be an L-datum, and let $c$ be an element of $X$ that is stably separated from the marked point of $X$. Then
\begin{displaymath}
\pi(L(X,X')) = \pi(L(X \setminus c, X' \setminus c)).
\end{displaymath}
\end{proposition}

\begin{proof}
Let $a \in X$ and $b \in X'$ be the marked points, and equate $X \setminus a$ with $X' \setminus b$; call this structure $W$. Thus $X=W \cup \{a\}$ and $X'=W \cup \{b\}$. Let $Y_1, \ldots, Y_n$ be the proper amalgamations of $W \cup \{a\}$ and $W \cup \{b\}$, i.e., those in which $a$ and $b$ are distinct. We assume the each $Y_i$ has underlying set $W \cup \{a,b\}$. Let $W'=W \setminus c$ and $Y'=Y_i \setminus c$. We claim that $Y'_1, \ldots, Y'_n$ are exactly the proper amalgamations of $W' \cup \{a\}$ and $W' \cup \{b\}$, i.e., every amalgamation appears on this list, and no two on the list are the same.

First, we show that the $Y'_i$ are distinct. Since $a$ and $c$ are stably separated in $W \cup \{a\}$, they are also separated in $Y_i$. Thus $Y_i$ is the unique amalgamation of $Y_i \setminus c=Y'_i$ and $Y_i \setminus a=W \cup \{b\}$. In other words, there is exactly one structure on $W \cup \{a,b\}$ (namely $Y_i$) that is compatible with the given structure on $W \cup \{b\}$ and with $Y'_i$. Thus if $Y'_i=Y'_j$ then we would have $Y_i=Y_j$, and thus $i=j$.

We now show that every proper amalgamation appears. Thus suppose that $Z'$ is some amalgamation of $W' \cup \{a\}$ and $W' \cup \{b\}$; we think of $Z'$ as a structure on the set $W' \cup \{a,b\}$. Consider the amalgamation diagram
\begin{displaymath}
\xymatrix{
Z' \ar[r] & ? \\
W' \cup \{b\} \ar[u] \ar[r] & W \cup \{b\} \ar[u] }
\end{displaymath}
Since $\fF$ is a Fra\"iss\'e class, there is some amalgamation $Z$. We regard $Z$ as structure on $W \cup \{a,b\}$ that restricts to the given structure on $W \cup \{b\}$, and to $Z'$ on $W' \cup \{a,b\}$. Consider the restriction of $Z$ to $W \cup \{a\}$. This restricts to the given structure on $W$ (since $Z$ already restricts to the given structure on $W \cup \{b\}$), and it restricts to the given structure on $W' \cup \{a\}$ (since $Z$ already restricts to the given structure on $W' \cup \{a,b\}$). Thus $Z \setminus b$ is an amalgamation of $W$ and $W' \cup \{a\}$. Since $a$ and $c$ are separated, $W \cup \{a\}$ is the unique such amalgamation, and so $Z$ restricts to the given structure on $W \cup \{a\}$. Thus $Z$ is in fact an amalgamation of $W \cup \{a\}$ and $W \cup \{b\}$, and therefore one of the $Y_i$'s, and so $Z'=Y'_i$, as required.

We have
\begin{displaymath}
L(W \cup \{a\}, W \cup \{b\}) = \lbb W \cup \{a\}, a \rbb - \delta - \sum_{i=1}^n \lbb Y_i, a \rbb
\end{displaymath}
while
\begin{displaymath}
L(W' \cup \{a\}, W' \cup \{b\}) = \lbb W' \cup \{a\}, a \rbb - \delta -  \sum_{i=1}^n \lbb Y'_i, a \rbb.
\end{displaymath}
Note that the two $\delta$'s above are the same. Since $a$ and $c$ are stably separated in $W \cup \{a\}$, they are also separated in each $Y_i$, and so we have
\begin{displaymath}
\pi(\lbb W \cup \{a\}, a \rbb) = \pi(\lbb W' \cup \{a\}, a \rbb), \qquad
\pi(\lbb Y_i, a \rbb) = \pi(\lbb Y'_i, a \rbb),
\end{displaymath}
and so the two linear forms have the same image under $\pi$.
\end{proof}

\begin{definition}
We make the following definitions:
\begin{enumerate}
\item An L-datum $(X,X')$ is \defn{separated} if there is a unique amalgamation of $X$ and $X'$, and this amalgamation is proper (the marked points remain distinct).
\item An L-datum $(X,X')$ is \defn{minimal} if it is not separated and there is no element of $X$ that is stably separated from the marked point.
\item We define $\fb_1^*$ to be the ideal of $P^*/\fb_0^*$ generated by the elements $\pi(L(X,X'))$ where $(X,X')$ is a minimal L-datum. \qedhere
\end{enumerate}
\end{definition}

\begin{proposition} \label{prop:Qmin}
Let $(X,a,b)$ be a Q-datum.
\begin{enumerate}
\item If $a$ and $b$ are separated then $\pi(Q(X,a,b))=0$.
\item Suppose there exists $c \in X$ distinct from $a$ and $b$ such that $c$ is extraneous in $(X,a)$, $(X,b)$, $(X \setminus a, b)$, and $(X \setminus b, a)$. Then $\pi(Q(X,a,b))=\pi(Q(X \setminus c, a, b))$.
\end{enumerate}
\end{proposition}

\begin{proof}
(a) Since $a$ and $b$ are separated, we have
\begin{displaymath}
\pi(\lbb X, a \rbb)=\pi(\lbb X \setminus b, a \rbb), \qquad
\pi(\lbb X, b \rbb)=\pi(\lbb X \setminus a, b \rbb),
\end{displaymath}
from which the claim immediately follows.

(b) This follows because the the four variables appearing in $Q(X, a, b)$ and the corresponding variables in $Q(X \setminus c, a, b)$ have the same images under $\pi$, e.g., $\lbb X, a \rbb$ and $\lbb X \setminus c, a \rbb$ have the same image.
\end{proof}

\begin{definition} \label{defn:Qmin}
We make the following definitions:
\begin{enumerate}
\item A Q-datum $(X,a,b)$ is \defn{minimal} if (i) it is not symmetric; (ii) $a$ and $b$ are not separated in $X$; and (iii) there is no $c \in X \setminus \{a,b,\}$ such that $c$ is extraneous in $(X,a)$, $(X,b)$, $(X \setminus a, b)$, and $(X \setminus b, a)$.
\item We let $\fb_2^*$ be the ideal of $P^*/\fb_0^*$ generated by the elements $\pi(Q(X,a,b))$ where $(X,a,b)$ is a minimal Q-datum. \qedhere
\end{enumerate}
\end{definition}

We let $\fb^*=\pi^{-1}(\fb_1^*)+\pi^{-1}(\fb_2^*)$, which is an ideal of $P^*$. Recall the homomorphism $\phi \colon P \to \Theta(\fF)$ from \S \ref{ss:Theta-present}. We let $\phi^* \colon P^* \to \Theta(\fF)$ be the restriction of $\phi$ to $P^* \subset P$. The following is the main result we are after. We emphasize that this result is due to the second author, and will appear in his thesis \cite{Nekrasov}.

\begin{theorem} \label{thm:refined}
The map $\phi^*$ induces an isomorphism $P^*/\fb^* \to \Theta(\fF)$.
\end{theorem}

\begin{proof}
We first claim that $\pi(\fb_1)=\fb_1^*$. Let $(X,X')$ be an L-datum. If there is an element $c \in X$ that is stably separated from the marked point, replace $(X,X')$ with $(X \setminus c, X' \setminus c)$. Iterate this process until we end up with $(Y,Y')$ in which there is no such point $c$. By Proposition~\ref{prop:Lmin}, $L(X,X')$ and $L(Y,Y')$ have the same image under $\pi$. Now, if $(Y,Y')$ is separated then $L(Y,Y')$ has the form $\lbb Y \rbb - \lbb Z \rbb$ where $Z$ is equivalent to $Y$, and so $L(Y,Y')$ belongs to $\fb_0^*$ and thus maps to~0 under $\pi$. Otherwise, $(Y,Y')$ is a minimal L-datum. We thus see that $L(X,X')$ either maps to~0 under $\pi$, or has the same image as $L(Y,Y')$ for some minimal L-datum $(Y,Y')$. Thus $\pi(\fb_1) \subset \fb_1^*$. The reverse containment is clear (every minimal L-datum produces an element of $\fb_1$).

We next claim that $\pi(\fb_2)=\fb_2^*$. Let $(X,a,b)$ be a Q-datum. If there is $c$ as in Definition~\ref{defn:Qmin}(a)(iii) then replace $(X,a,b)$ with $(X \setminus c, a, b)$. Iterate this process until we end up with $(Y,a,b)$ in which there is no such $c$. By Proposition~\ref{prop:Qmin}, $Q(X,a,b)$ and $Q(Y,a,b)$ have the same image under $\pi$. Now, if $(Y,a,b)$ is symmetric or $a$ and $b$ are not separated in $Y$ then $Q(Y,a,b)$ maps to~0 under $\pi$. Otherwise, $(Y,a,b)$ is minimal. We thus see that $Q(X,a,b)$ either maps to~0 under $\pi$, or has the same image as $Q(Y,a,b)$ for some minimal Q-datum $(Y,a,b)$. Thus $\pi(\fb_2) \subset \fb_2^*$. The reverse containment is clear.

From the above two paragraphs, we see that $\pi$ induces a map $P/\fb \to P^*/\fb^*$, which is clearly surjective. Recall (Proposition~\ref{prop:Theta-present}) that $\phi \colon P \to \Theta(\fF)$ induces an isomorphism $P/\fb \to \Theta(\fF)$. It is clear that $\phi^*$ induces a homomorphism $P^*/\fb^* \to \Theta(\fF)$, which is surjective since $\Theta(\fF)$ is generated by minimal marked structures. One easily sees that the following diagram commutes
\begin{displaymath}
\xymatrix@C=6ex{
P/\fb \ar[d]_{\pi} \ar[rd]^{\phi} \\
P^*/\fb^* \ar[r]^{\phi^*} & \Theta(\fF) }
\end{displaymath}
It follows that the bottom map here is injective, and thus an isomorphism. The result follows.
\end{proof}

\begin{example} \label{ex:ordered}
Let $\fF$ be the Fr\"aiss\'e class of finite totally ordered sets. From the description of separated in vertices in Example~\ref{example:DLO}, it is clear that that ``separated'' and ``stably separated'' coincide. Therefore if the L-datum $(X,X')$ is minimal then cardinality of $X$ (and $X'$) is at most two and $X$ is isomorphic to $X'$. So $\fb_1^*$ is generated by four equations
\begin{align*}
    L(O_1)&: \, O_1 = 1 + O_r + O_l\\
    L(O_r)&: \, O_r = 1 + O_r + O_m\\
    L(O_l)&: \, O_l = 1 + O_l + O_m\\
    L(O_m)&: \, O_m = 1 + 2 O_m
\end{align*}
After simplifications, we see that
\begin{displaymath}
\fb_1^* = \langle 1 - O_1 + O_r + O_l, O_m +1 \rangle.
\end{displaymath}
Analogously, if Q-datum $(X,a,b)$ is minimal then $a$ and $b$ are consecutive elements in $X$ and cardinality of $X$ is at most 3. So minimal Q-data are 
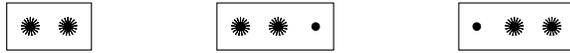
\begin{displaymath}
\qquad\qquad
\tikz[framed]{
\node[chosen] (B) at (0,0) {};
\node[chosen] (B) at (.5,0) {};
}
\qquad\qquad
\tikz[framed]{
\node[chosen] (B) at (0,0) {};
\node[chosen] (B) at (.5,0) {};
\node[leaf] (B) at (1,0) {};
}
\qquad\qquad
\tikz[framed]{
\node[leaf] (B) at (0,0) {};
\node[chosen] (B) at (.5,0) {};
\node[chosen] (B) at (1,0) {};
}
\end{displaymath}
Note that the first Q-datum here is \emph{not} symmetric. Therefore
\begin{displaymath}
\fb_2^* = \langle O_1 (O_r - O_l), O_r (O_r - O_m), O_l (O_l - O_m) \rangle.
\end{displaymath}
We thus find $\Theta(\fF) \cong \bZ^4$ as a ring. See \cite[\S 17]{repst} for a different derivation, and some further details.
\end{example}

\subsection{The induced regular measure} \label{ss:induced}

To close this section, we explain how every field valued measure induces a regular measure on some Fra\"iss\'e subclass. We will use this observation in \S \ref{ss:finite-meas} to construct measures on trees of finite level. Let $\fF$ be a Fra\"iss\'e class, and let $\mu$ be a measure for $\fF$ valued in a domain $k$. Let $\fF'$ be the subclass of $\fF$ consisting of structures $X$ for which $\mu(X)$ is non-zero.

\begin{proposition} \label{prop:induced}
The class $\fF'$ is Fra\"iss\'e, and $\mu$ restricts to a measure $\mu'$ on $\fF'$ that is nowhere zero. If $k$ is a field then $\mu'$ is regular.
\end{proposition}

\begin{proof}
It is clear that $\fF'$ has only finitely many isomorphism classes of $n$-element structures. Suppose $X \in \fF'$ and $Y$ embeds into $X$. We have
\begin{displaymath}
\mu(X)=\mu(Y) \cdot \mu(Y \to X).
\end{displaymath}
Since $\mu(X)$ is non-zero and $k$ is a domain, it follows that $\mu(Y) \ne 0$. Thus $Y$ belongs to $\fF'$, which shows that $\fF$ is hereditary. Finally, suppose that we have embeddings $X \to Y$ and $X \to X'$ in $\fF'$, and let $Y'_1, \ldots, Y'_n$ be the various amalgamations in $\fF$. Then
\begin{displaymath}
\mu(X \to Y) = \sum_{i=1}^n \mu(X' \to Y'_i).
\end{displaymath}
Since $X$ and $Y$ belong to $\fF$, it follows that the left side is non-zero. Thus there is some $i$ for which $\mu(X' \to Y'_i)$ is non-zero. Since $X'$ also belongs to $\fF$, we see that $\mu(Y'_i)$ is non-zero, and so $Y'_i \in \fF'$. Thus $\fF'$ satisfies the amalgamation property.

Let $\mu'$ be the restriction of $\mu$ to $\fF'$. It is clear that $\mu'$ satisfies conditions (a) and (b) of Definition~\ref{defn:meas}. We now verify condition (c). Let $X \to Y$ and $X \to X'$ be embeddings in $\fF'$, and let $Y'_1, \ldots, Y'_n$ be the various amalgamations in $\fF$, labeled so that $Y'_1, \ldots, Y'_m$ belong to $\fF'$ and $Y'_{m+1}, \dots, Y'_n$ do not belong to $\fF'$. Since $\mu$ is a measure, we have
\begin{displaymath}
\mu(X \to Y) = \sum_{i=1}^n \mu(X' \to Y'_i).
\end{displaymath}
For $m<i\le n$, we have $\mu(X' \to Y'_i)=0$, since $\mu(X') \ne 0$ and $\mu(Y'_i)=0$. Thus these terms can be omitted from the sum. This verifies (c) for $\mu'$.

If $i \colon X \to Y$ is an embedding in $\fF'$ then we have $\mu(X) \cdot \mu(i) = \mu(Y)$. Since $k$ is a domain and $\mu(Y)$ is non-zero, it follows that $\mu(i)$ is non-zero. Thus $\mu$ takes only non-zero values. In particular, if $k$ is a field then $\mu$ is regular.
\end{proof}

%\begin{remark}
%The proposition shows importance of classification of all measures on a class $\fF$ from a model-theoretic side: non-regular measures classify $\omega$-categorical substructures inside some fixed $\omega$-categorical structure. 
%\end{remark}

\section{Measures on trees} \label{s:trees}

\subsection{Overview}

In this section, we study measures on trees. These measures are the crucial ingredient in our construction of arboreal tensor categories. We briefly outline the contents of this section. In \S \ref{ss:trees} we define the Fra\"iss\'e classes $\fT$ and $\fT_n$ of trees that we will use. In \S \ref{ss:treesep} we characterize separated vertices in trees, and in \S \ref{ss:mintree} we classify the minimal marked trees; this gives us a generating set for $\Theta(\fT)$. In \S \ref{ss:theta-tree} we determine the relations amongst these generators using the method of Theorem~\ref{thm:refined}. This presentation shows that there are two 1-parameter families of measures on $\fT$. We will focus on one of these families in the remainder of the paper, since we are unable to say much about the other one. In \S \ref{ss:goodmeas}, we define this family of measures (denoted $\mu_t$) and compute it explicitly. In \S \ref{ss:finite-meas} we show how $\mu_n$ (for $n \ge 3$ an integer) induces a measure on $\fT_n$, and in \S \ref{ss:inf-meas} we look at $\mu_t$ as $t \to \infty$.

\subsection{Trees} \label{ss:trees}

A \defn{tree} is a connected simple graph without cycles; a simple graph is an undirected graph without loops or parallel edges, having finitely many vertices. Let $T$ be a tree. We say that a vertex of $T$ is a \defn{leaf} if it has valence one, and a \defn{node} otherwise. For a leaf $x$, we write $N(x)$ for the unique node $x$ is connected to, assuming $T$ has at least three leaves (which ensures this node exists). We say that $T$ is \defn{reduced} if it has no nodes of valence two. Any tree has a \defn{reduction}, obtained by deleting each node of valence two (and joining the two neighbors by an edge).

Let $T$ be a tree. We define a quaternary relation $R$ on the set $L(T)$ of leaves of $T$ as follows. Let $x_1$, $x_2$, $y_1$, and $y_2$ be leaves. Then $R(x_1,x_2;y_1,y_2)$ is true if and only if the four leaves are distinct, and the shortest path joining $x_1$ and $x_2$ has a common edge with the shortest path joining $y_1$ and $y_2$. In this way, $L(T)$ is a finite relations structure, having one relation of arity four. If $T'$ is a second tree then the structures $L(T)$ and $L(T')$ are isomorphic if and only if the reductions of $T$ and $T'$ are isomorphic \cite[Proposition~3.1]{CameronTrees}.

We define a \defn{tree structure} to be a finite set with a quaternary relation that is isomorphic to $L(T)$ for some tree $T$. We let $\fT$ be the class of tree structures. This is a Fra\"iss\'e class \cite[Proposition~3.2]{CameronTrees}. We define the \defn{level} of a tree $T$ (or the corresponding structure) to be the maximum valence of a node in $T$. Fix an integer $n \ge 3$. We let $\fT_n$ be the class of tree structures of level $\le n$. This is also a Fra\"iss\'e class; see the ``First variation'' after \cite[Proposition~3.2]{CameronTrees}.

For additional background on these Fra\"iss\'e classes, see \cite{CameronTrees}. The case $\fT_3$ is also discussed in \cite[\S 2.6]{CameronBook}, where the relevant trees are called ``boron trees.''

\begin{remark}
The Fra\"iss\'e limits of tree structures (see \S \ref{ss:oligo}), of either bounded or unbounded valance, are interesting geometric objects in their own right.  For any two vertices in the ``tree" there is a unique path between them,  but the tree branches along the path at a dense set of points. Any finite collection of vertices will span a genuine finite tree, but there is no notion of adjacency for vertices in the full tree.  We do not explicitly use this geometric description, but it motivates some of the calculations in \S \ref{s:example}.
\end{remark}

\subsection{Separated vertices} \label{ss:treesep}

Recall that two leaves $a$ and $b$ in a tree $T$ are \defn{separated} if $T$ is the unique amalgamation of $T \setminus a$ and $T \setminus b$ over $T \setminus \{a,b\}$. The following proposition gives a concrete characterization of such elements.

\begin{proposition} \label{prop:treesep}
Let $T$ be a reduced tree and let $a$ and $b$ be distinct leaves of $T$. Then $a$ and $b$ are separated if and only if one of the following two conditions hold:
\begin{enumerate}
\item $N(a)$ and $N(b)$ are distinct and not neighbors.
\item $N(a)$ and $N(b)$ are distinct, and at least one has valence $\ge 4$.
\end{enumerate}
\end{proposition}

\begin{proof}
We explain how to show that (a) implies $a$ and $b$ are separated; the idea for (b) is similar. Let $T'$ be an amalgamation of $T \setminus a$ and $T \setminus b$ over $T \setminus \{a,b\}$. It is easy to see using condition (a) that $a$ and $b$ remain distinct in $T'$. Thus we may as well equate the leaf sets $L(T)$ and $L(T')$; call this common set $L$. The two tree structures amount to two quaternary relations $R$ and $R'$ on $L$, and we must show $R=R'$. We are given that $R$ and $R'$ agree on all 4-tuples in which $a$ and $b$ do not both appear.

It is easy to find $c_1$, $c_2$, and $c_3$ in $L \setminus \{a,b\}$ witnessing (a) in both structures; that is, if we look at the two reduced trees on $\{a,b,c_1,c_2,c_3\}$ then $N(a)$ and $N(b)$ are distinct and not neighbors in each one. Now, let $x,y \in L$ be given. Suppose we want to verify
\begin{displaymath}
R(a,b;x,y) \iff R'(a,b;x,y).
\end{displaymath}
We can verify this by considering the restrictions of $R$ and $R'$ to $\{a,b,c_1,c_2,c_3,x,y\}$. Similarly for other instances of $R$ and $R'$. This shows that it suffices to prove (a) for trees with $\le 7$ leaves, and this is straightforward.

We now handle the opposite direction: supposing (a) and (b) both fail, we show that $a$ and $b$ are not separated. First suppose that $N(a)=N(b)$. Then there is an isomorphism $i \colon T\setminus a \to T\setminus b$ that is the identity on $T \setminus \{a,b\}$, and this realizes $T \setminus a$ as an amalgamation of $T \setminus a$ and $T \setminus b$ over $T \setminus \{a,b\}$. Hence $T$ is not the unique such amalgamation, and so $a$ and $b$ are not separated.

Now suppose $N(a) \ne N(b)$. Since (a) and (b) both fail, it follows that $N(a)$ and $N(b)$ are neighbors, and each has valence $\le 3$ (so $ = 3$ as $T$ is reduced). The picture is
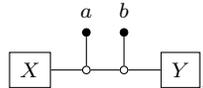
\begin{displaymath}
\tikz{
\node[boron] (A) at (0,0) {};
\node[boron] (B) at (.5,0) {};
\node[leaf,label={\tiny $a$}] (a) at (0,.5) {};
\node[leaf,label={\tiny $b$}] (b) at (.5,.5) {};
\node[draw] (x) at (-.75,0) {\tiny $X$};
\node[draw] (y) at (1.25,0) {\tiny $Y$};
\path[draw] (A)--(B);
\path[draw] (A)--(x);
\path[draw] (B)--(y);
\path[draw] (A)--(a);
\path[draw] (B)--(b);
}.
\end{displaymath}
The leaves $a$ and $b$ are not separated, as
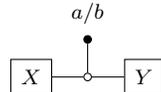
\begin{displaymath}
\tikz{
\node[boron] (A) at (0,0) {};
\node[leaf,label={\tiny $a/b$}] (a) at (0,.5) {};
\node[draw] (x) at (-.75,0) {\tiny $X$};
\node[draw] (y) at (.75,0) {\tiny $Y$};
\path[draw] (A)--(x);
\path[draw] (A)--(y);
\path[draw] (A)--(a);
}
\end{displaymath}
is also an amalgamation of $T \setminus a$ and $T \setminus b$ over $T \setminus \{a,b\}$.
\end{proof}

We note the following corollary of the proposition. It would be useful to know when this corollary holds for more general Fra\"iss\'e classes.

\begin{corollary} \label{cor:treesep}
For trees, ``separated'' and ``stably separated'' coincide.
\end{corollary}

\subsection{Minimal marked trees} \label{ss:mintree}

Let $T$ be a marked tree; we denote the marked leaf by $\ast$. Assuming $T$ has at least four leaves, we assign it one of the following types:
\begin{itemize}
\item[(I${}_m$)] If $N(\ast)$ has valence $m \ge 4$.
\item[(II)] If $N(\ast)$ has valence three and two neighbors are leaves.
\item[(III)] If $N(\ast)$ has valence three and exactly one neighbor is a leaf.
\end{itemize}
If $T$ has $m \le 3$ leaves then we assign $T$ type I${}_m$.

We let $X_{m+1}$ (for $m \ge 2$), $Y$, and $Z$ denote the following marked trees
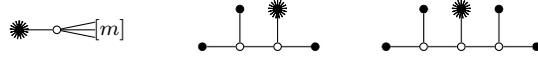
\begin{displaymath}
%\tikz{
%\node[chosen] (x) at (0,0) {};
%}
%
%\tikz{
%\node[chosen] (x) at (0,0) {};
%\node[leaf] (y) at (.5,0) {};
%\path[draw] (x)--(y);
%}
%
\tikz{
\node[boron] (A) at (0,0) {};
\node[chosen] (x) at (-0.5,0) {};
\node[mid] (a1) at (.5,-.1) {};
\node[mid] (a2) at (.5,0) {};
\node[mid] (a3) at (.5,.1) {};
\node () at (.7, 0) {\tiny $[m]$};
\path[draw] (A)--(x);
\path[draw] (A)--(a1);
\path[draw] (A)--(a2);
\path[draw] (A)--(a3);
}
\qquad
\tikz{
\node[boron] (A) at (0,0) {};
\node[boron] (B) at (0.5,0) {};
\node[leaf] (x) at (-0.5,0) {};
\node[leaf] (y) at (1,0) {};
\node[leaf] (a) at (0,0.5) {};
\node[chosen] (b) at (0.5,0.5) {};
\path[draw] (A)--(B);
\path[draw] (A)--(x);
\path[draw] (B)--(y);
\path[draw] (A)--(a);
\path[draw] (B)--(b);
}
\qquad
\tikz{
\node[boron] (A) at (0,0) {};
\node[boron] (B) at (0.5,0) {};
\node[boron] (C) at (1,0) {};
\node[leaf] (x) at (-0.5,0) {};
\node[leaf] (y) at (1.5,0) {};
\node[leaf] (a) at (0,0.5) {};
\node[chosen] (b) at (0.5,0.5) {};
\node[leaf] (c) at (1,0.5) {};
\path[draw] (A)--(B);
\path[draw] (B)--(C);
\path[draw] (A)--(x);
\path[draw] (C)--(y);
\path[draw] (A)--(a);
\path[draw] (B)--(b);
\path[draw] (C)--(c);
}
\end{displaymath}
To be clear, $X_{m+1}$ has $m+1$ leaves total, and they all connect to a common node. We also let $X_m$ for $m=1,2$ be the unique marked tree with $m$ vertices. We note that $X_m$ has type I${}_m$ for any $m \ge 1$, while $Y$ has type II and $Z$ has type III.

\begin{proposition} \label{prop:minimal}
We have the following:
\begin{enumerate}
\item The minimal marked trees are exactly $X_m$ (for $m \ge 1$), $Y$, and $Z$.
\item Two marked trees are equivalent if and only if they have the same type.
\item There are no non-trivial equivalences among the minimal marked trees.
\end{enumerate}
\end{proposition}

\begin{proof}
(a) Let $T$ be a minimal marked tree with $\ge 4$ leaves. If $N(\ast)$ has valence $\ge 4$ then there cannot be any other nodes in $T$ by minimality, and so $T$ is isomorphic to some $X_m$. Now suppose $N(\ast)$ has valence three. Any other node must be adjacent to $N(\ast)$ and of valence three, by minimality. Thus $T$ is isomorphic to $Y$ or $Z$.

(b) We first show that equivalent marked trees have the same type. It suffices to show that removing an extraneous leaf does not change the type. Thus let $T$ be a marked tree and let $c$ be an extraneous leaf. First suppose $T$ has type I${}_m$, i.e., $N(\ast)$ has valence $m \ge 4$. Since $N(c) \ne N(\ast)$, it follows that $N(\ast)$ still has valence $m$ in $T \setminus c$, and so $T \setminus c$ also has type I${}_m$. Now suppose $T$ has type~II, i.e., $N(\ast)$ has valence three and two of its neighbors are leaves. Since $c$ is not one of its neighbors, this remains true after deleting $c$, and so $T \setminus c$ has type~II. Finally, suppose $T$ has type~III, so that $N(\ast)$ has valence three and $\ast$ is its unique leaf neighbor. If $N(c)$ is one of the neighbors of $N(\ast)$ then it must have valence $\ge 4$ for $c$ to be separated from $\ast$, and so deleting $c$ will not delete the node $N(c)$. Thus $T \setminus c$ still has type~III. If $N(c)$ is not one of the neighbors of $N(\ast)$ then deleting $c$ will not convert any of the nodes neighboring $N(\ast)$ into leaves, and so again $T \setminus c$ still has type~III.

We now show that marked trees of the same type are equivalent. Thus suppose that $X$ and $Y$ have the same type. Let $X'$ and $Y'$ be minimal marked trees equivalent to $X$ and $Y$. By the above paragraph, $X'$ and $Y'$ have the same common type as $X$ and $Y$. From the classification of minimal marked trees in (a), and the computation of their types, we see that $X'=Y'$. Thus $X$ and $Y$ are equivalent.

(c) This follows from (a) and (b) since the minimal marked trees have distinct types.
\end{proof}

\subsection{The $\Theta$ ring} \label{ss:theta-tree}

We now determine the measures for $\fT$. Write $x_m$, $y$, and $z$ for the classes of $X_m$, $Y$, and $Z$ in $\Theta(\fT)$. The following is our main result:

\begin{theorem} \label{thm:Theta}
We have an isomorphism
\begin{displaymath}
i \colon \Theta(\fT) \to \bZ[u,v]/(uv)
\end{displaymath}
satisfying
\begin{align*}
i(x_1) &= u+v+2 & i(x_m) &= v+1-(m-2)(u+1) \\
i(x_2) &= u+v+1 & i(y) &= u \\
i(x_3) &= u+v & i(z) &= u-v,
\end{align*}
where here $m \ge 4$.
\end{theorem}

We thus see that $\Spec(\Theta(\fT))$ is a union of two lines, meeting at a point. In other words, there are two 1-parameter families of measures for $\fT$. The family with $v=0$ is generically regular, and will be the measure of we focus on in this paper; see \S \ref{ss:goodmeas} for more. The family with $u=0$ is nowhere regular. We say nothing more about it in this paper, but it would be an interesting topic for further study.

To prove Theorem~\ref{thm:Theta}, we use the method of \S \ref{ss:refined}. We have already seen that there are no non-trivial equivalences amongst minimal marked trees, and so the ideal $\fb_0^*$ vanishes. We now determine the remaining ideals appearing in this presentation.

\begin{proposition}
The minimal L-data are the duplicates of the minimal marked trees.
\end{proposition}

\begin{proof}
Every minimal $L$-datum is of the form $(T, T', i)$, where $T$ is a minimal marked tree, $T'$ is obtained from $T$ by deleting the marked vertex $a$ and adding in a new marked vertex $b$ (possibly in the same place), and $i \colon T \setminus a \to T' \setminus b$ is the identity; this follows from the fact that ``separated'' and ``stably separated'' are the same (Corollary~\ref{cor:treesep}). Thus $T$ is one of $X_m$, $Y$, or $Z$ by Proposition~\ref{prop:minimal}. It is now a simple matter to verify the lemma.

We explain one case in detail. Suppose $T=X_{m+1}$ with $m \ge 3$. Ignoring minimality, there are two possibilities for $T'$:
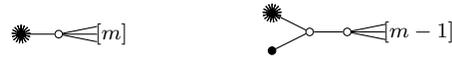
\begin{displaymath}
\tikz{
\node[boron] (A) at (0,0) {};
\node[chosen] (x) at (-0.5,0) {};
\node[mid] (b1) at (.5,-.1) {};
\node[mid] (b2) at (.5,0) {};
\node[mid] (b3) at (.5,.1) {};
\node () at (.7, 0) {\tiny $[m]$};
\path[draw] (A)--(x);
\path[draw] (A)--(b1);
\path[draw] (A)--(b2);
\path[draw] (A)--(b3);
}
\qquad\qquad
\tikz{
\node[boron] (A) at (0,0) {};
\node[boron] (B) at (.5,0) {};
\node[chosen] (x) at (-.5,.25) {};
\node[leaf] (y) at (-.5,-.25) {};
\node[mid] (b1) at (1,-.1) {};
\node[mid] (b2) at (1,0) {};
\node[mid] (b3) at (1,.1) {};
\node () at (1.45, 0) {\tiny $[m-1]$};
\path[draw] (A)--(B);
\path[draw] (A)--(x);
\path[draw] (A)--(y);
\path[draw] (B)--(b1);
\path[draw] (B)--(b2);
\path[draw] (B)--(b3);
}
\end{displaymath}
In the first case, $T'$ is simply $T$, and this is one of the minimal L-data appearing in the statement of the lemma. In the second case, the L-datum is separated (since the node in $T$ has valence $\ge 4$), and thus not minimal.
\end{proof}

\begin{corollary} \label{cor:linear}
The ideal $\fb_1^*$ is generated by the following linear forms:
\begin{align*}
& 1+x_2-x_1, & & 1+x_{m+1}-x_m+y, \\
& 1+x_3-x_2, & & 1+z+y+x_4, \\
& 1+x_4-x_3+3y,
\end{align*}
where $m \ge 4$.
\end{corollary}

\begin{proof}
The forms in the left column are $L(X_m)$ for $1 \le m \le 3$. The first form in the right column is $L(X_m)$ with $m \ge 4$. The second form in the right column is $L(Y)$, which, it turns out, coincides with $L(Z)$.

We explain the computation of $L(X_3)$ in detail; the remaining cases are similar. For this, it is useful to introduce two copies of the tree $X_3$
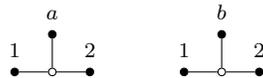
\begin{displaymath}
\tikz{
\node[boron] (A) at (0,0) {};
\node[leaf,label={\tiny 1}] (x) at (-0.5,0) {};
\node[leaf,label={\tiny 2}] (y) at (.5,0) {};
\node[leaf,label={\tiny $a$}] (a) at (0,.5) {};
\path[draw] (A)--(x);
\path[draw] (A)--(y);
\path[draw] (A)--(a);
}
\qquad
\tikz{
\node[boron] (A) at (0,0) {};
\node[leaf,label={\tiny 1}] (x) at (-0.5,0) {};
\node[leaf,label={\tiny 2}] (y) at (.5,0) {};
\node[leaf,label={\tiny $b$}] (a) at (0,.5) {};
\path[draw] (A)--(x);
\path[draw] (A)--(y);
\path[draw] (A)--(a);
}
\end{displaymath}
Here $a$ and $b$ are the marked vertices. The self-amalgamations are the trees with vertices labeled 1, 2, $a$, $b$, such that when we delete $a$ or $b$ we obtain the above labeled trees; note that $a$ and $b$ are allowed to label the same vertex. The amalgamations are as follows:
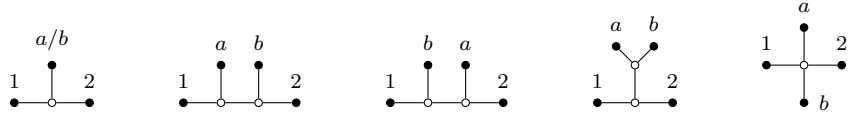
\begin{displaymath}
\tikz[baseline=0pt]{
\node[boron] (A) at (0,0) {};
\node[leaf,label={\tiny 1}] (x) at (-0.5,0) {};
\node[leaf,label={\tiny 2}] (y) at (.5,0) {};
\node[leaf,label={\tiny $a/b$}] (a) at (0,.5) {};
\path[draw] (A)--(x);
\path[draw] (A)--(y);
\path[draw] (A)--(a);
}
\qquad
\tikz[baseline=0pt]{
\node[boron] (A) at (0,0) {};
\node[boron] (B) at (.5,0) {};
\node[leaf,label={\tiny 1}] (x) at (-0.5,0) {};
\node[leaf,label={\tiny 2}] (y) at (1,0) {};
\node[leaf,label={\tiny $a$}] (a) at (0,.5) {};
\node[leaf,label={\tiny $b$}] (b) at (.5,.5) {};
\path[draw] (A)--(B);
\path[draw] (A)--(x);
\path[draw] (A)--(a);
\path[draw] (B)--(b);
\path[draw] (B)--(y);
}
\qquad
\tikz[baseline=0pt]{
\node[boron] (A) at (0,0) {};
\node[boron] (B) at (.5,0) {};
\node[leaf,label={\tiny 1}] (x) at (-0.5,0) {};
\node[leaf,label={\tiny 2}] (y) at (1,0) {};
\node[leaf,label={\tiny $b$}] (a) at (0,.5) {};
\node[leaf,label={\tiny $a$}] (b) at (.5,.5) {};
\path[draw] (A)--(B);
\path[draw] (A)--(x);
\path[draw] (A)--(a);
\path[draw] (B)--(b);
\path[draw] (B)--(y);
}
\qquad
\tikz[baseline=0pt]{
\node[boron] (A) at (0,0) {};
\node[boron] (B) at (0,.5) {};
\node[leaf,label={\tiny 1}] (x) at (-0.5,0) {};
\node[leaf,label={\tiny 2}] (y) at (.5,0) {};
\node[leaf,label={\tiny $a$}] (a) at (-.25,.75) {};
\node[leaf,label={\tiny $b$}] (b) at (.25,.75) {};
\path[draw] (A)--(B);
\path[draw] (A)--(x);
\path[draw] (B)--(a);
\path[draw] (B)--(b);
\path[draw] (A)--(y);
}
\qquad
\tikz[baseline=0pt]{
\node[boron] (A) at (0,.5) {};
\node[leaf,label={\tiny 1}] (x) at (-0.5,.5) {};
\node[leaf,label={\tiny 2}] (y) at (.5,.5) {};
\node[leaf,label={\tiny $a$}] (a) at (0,1) {};
\node[leaf,label=right:{\tiny $b$}] (b) at (0,0) {};
\path[draw] (A)--(x);
\path[draw] (A)--(a);
\path[draw] (A)--(b);
\path[draw] (A)--(y);
}
\end{displaymath}
For each of the five trees above, we can consider the subtree formed by the vertices 1, 2, and $b$. Letting $i_1, \ldots, i_5$ be the inclusions, the amalgamation equation states that $[X_3] = [i_1]+\cdots+[i_5]$. Now, $i_1$ is the identity, and so $[i_1]=1$. The classes $[i_2]$, $[i_3]$, $[i_4]$ are each equal to $y$, while $[i_5]=x_4$. We thus obtain the equation $x_3=1+3y+x_4$.
\end{proof}

\begin{proposition}
The following is the unique minimal Q-datum
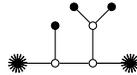
\begin{displaymath}
\tikz{
\node[boron] (A) at (0,0) {};
\node[boron] (B) at (.5,0) {};
\node[boron] (F) at (.5,.5) {};
\node[chosen] (x) at (-0.5,0) {};
\node[chosen] (y) at (1,0) {};
\node[leaf] (a) at (0,0.5) {};
\node[leaf] (f1) at (.25, .75) {};
\node[leaf] (f2) at (.75, .75) {};
\path[draw] (A)--(B);
\path[draw] (B)--(C);
\path[draw] (B)--(F);
\path[draw] (A)--(x);
\path[draw] (B)--(y);
\path[draw] (A)--(a);
\path[draw] (F)--(f1);
\path[draw] (F)--(f2);
}
\end{displaymath}
\end{proposition}

\begin{proof}
Let $(T,a,b)$ be a minimal Q-datum. Consider the path joining $a$ and $b$ in $T$. If this path contains no nodes then $a$ and $b$ are the unique vertices of $T$, and so the datum is symmetric and therefore not minimal. If the path contains only one node, then $N(a)=N(b)$, and so we are also in the symmetric case. If the path contains more than two nodes then $N(a)$ and $N(b)$ are not neighbors, so $a$ and $b$ are separated, and so the datum is not minimal. We conclude that the path must have exactly two nodes, namely $N(a)$ and $N(b)$. If either of these nodes has valence $\ge 4$ then $a$ and $b$ are separated, and so again the datum is not minimal. Thus each has valence three.

We thus see that $T$ can be drawn as
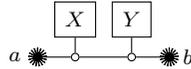
\begin{displaymath}
\tikz{
\node[boron] (A) at (0,0) {};
\node[boron] (B) at (.75,0) {};
\node[chosen] (x) at (-0.5,0) {};
\node[chosen] (y) at (1.25,0) {};
\node[draw] (a) at (0,.5) {\tiny $X$};
\node[draw] (b) at (.75,.5) {\tiny $Y$};
\path[draw] (A)--(B);
\path[draw] (A)--(x);
\path[draw] (B)--(y);
\path[draw] (A)--(a);
\path[draw] (B)--(b);
\node at (-.8, 0) {\tiny $a$};
\node at (1.5, 0) {\tiny $b$};
}
\end{displaymath}
where $X$ and $Y$ are subtrees. Now, if $X$ has more than one node then there is some vertex $c$ in $X$ such that $N(c)$ is not a neighbor of $N(a)$. Thus $c$ is separated from $a$ in $X \setminus b$ and from $b$ in $X \setminus a$, and so the datum is not minimal. We conclude that $X$ contains at most one node, and similarly for $Y$ (by symmetry). If $X$ has more than two vertices then its unique node has valence $\ge 4$, and so any vertex of $X$ is separated from $a$ in $X \setminus b$ and from $b$ in $X \setminus a$, and so the datum is not minimal. We conclude that $X$ has at most two vertices, and similarly for $Y$ (by symmetry). If $X$ and $Y$ are isomorphic then the datum is symmetric, and thus not minimal. We conclude that one of $X$ and $Y$ has one vertex, and the other has two, which yields the stated result.
\end{proof}

\begin{corollary} \label{cor:quadratic}
The ideal $\fb_2^*$ is generated by $y(z-y)$.
\end{corollary}

\begin{proof}[Proof of Theorem~\ref{thm:Theta}]
Let $P^*$ be the polynomial ring in variables $x_i$, for $i \ge 1$, and $y$ and $z$. Let $\fb^*=\fb_1^*+\fb_2^*$; this ideal is generated by the linear forms listed in Corollary~\ref{cor:linear} and the single quadratic form in Corollary~\ref{cor:quadratic}. By Theorem~\ref{thm:refined}, we have a natural isomorphism $\Theta(\fT)=P^*/\fb^*$. The result now follows from some straightforward algebra.
\end{proof}

\subsection{The primary measure} \label{ss:goodmeas}

Let $R=\bZ[\tau,(\tau-1)^{-1}]$. There is a unique ring homomorphism
\begin{displaymath}
\bZ[u,v]/(uv) \to R, \qquad u \mapsto -\frac{\tau-2}{\tau-1}, \qquad v \mapsto 0
\end{displaymath}
Let $\mu_{\tau} \colon \Theta(\fT) \to R$ be the $R$-valued measure corresponding to this homomorphism via Theorem~\ref{thm:Theta}. This measure satisfies
\begin{displaymath}
\mu_{\tau}(x_1) = \frac{\tau}{\tau-1}, \qquad
\mu_{\tau}(x_2) = \frac{1}{\tau-1}, \qquad
\mu_{\tau}(x_m) = \frac{\tau+1-m}{\tau-1}
\end{displaymath}
\begin{displaymath}
\mu_{\tau}(x_3) = \mu_{\tau}(y) = \mu_{\tau}(z) =  -\frac{\tau-2}{\tau-1},
\end{displaymath}
where $m \ge 4$ in the above equation. We now determine the formula for $\mu_{\tau}$ on an arbitrary tree. We first introduce some notation. For an integer $n \ge 3$, we define the following slight variant of the falling factorial:
\begin{displaymath}
[\tau]_n = (\tau-n+1) \cdots (\tau-2).
\end{displaymath}
For a tree $T$, write $\ell(T)$ for the number of leaves and $n(T)$ for the number of nodes. For a node $x$, we write $v(x)$ for its valence.

\begin{proposition} \label{prop:mu-formula}
For a tree $T$, we have
\begin{displaymath}
\mu_{\tau}(T) = \frac{(-1)^{n(T)} \cdot \tau}{(\tau-1)^{\ell(T)}} \cdot \prod_x [\tau]_{v(x)}
\end{displaymath}
where the product is over nodes $x$.
\end{proposition}

\begin{proof}
Put $\mu=\mu_{\tau}$ for this proof, and let $\mu'(T)$ denote the right side of the above formula; we must show $\mu(T)=\mu'(T)$. By definition, $\mu(T)$ is the measure of the embedding $\emptyset \to T$. If $a \in T$ is any leaf, then this embedding factors as $\emptyset \to T \setminus a \to T$, and so we have
\begin{displaymath}
\mu(T) = \mu(T \setminus a) \cdot \mu(T,a),
\end{displaymath}
where the second factor on the right is the value of $\mu$ on $[T,a] \in \Theta(\fT)$. By induction on the number of leaves, we can assume $\mu(T \setminus a)=\mu'(T \setminus a)$. It thus suffices to show
\begin{displaymath}
\mu'(T) = \mu'(T \setminus a) \cdot \mu(T,a),
\end{displaymath}
We consider two cases.

First suppose $N(a)$ has valence $m \ge 4$. Then $T \setminus a$ has the same number of nodes as $T$, but the valence of $N(a)$ has been reduced by one. Thus
\begin{displaymath}
\frac{\mu'(T)}{\mu'(T \setminus a)} = \frac{\tau-m+1}{\tau-1}
\end{displaymath}
On the other hand, the marked tree $(T,a)$ has type I${}_m$, and so $(T,a)$ is equivalent to $X_m$. Thus $\mu(T,a)=\mu(x_m)$, which agrees with the above.

Next suppose that $N(a)$ has valence three. Then $N(a)$ is removed when we delete $a$; all other nodes remain, and maintain the same valence. Thus
\begin{displaymath}
\frac{\mu'(T)}{\mu'(T \setminus a)} = -\frac{\tau-2}{\tau-1}
\end{displaymath}
On the other hand, the marked tree $(T,a)$ has type~II or~III, and is thus equivalent to $Y$ or $Z$. As $\mu(y)=\mu(z)$ agrees with the above, the result follows.
\end{proof}

\begin{remark}
We alternatively have the formula
\begin{displaymath}
\mu_{\tau}(T) = \frac{(-1)^{n(T)} \cdot \tau}{(\tau-1)^{\ell(T)}} \cdot \prod_{3 \le i} (\tau-i+1)^{n_i(T)},
\end{displaymath}
where $n_i(T)$ is the number of nodes of valence at least $i$.
\end{remark}

Let $k$ be an arbitrary ring, and let $t \in k$ be an element such that $t-1$ is a unit. We then have a homomorphism $R \to k$ via $\tau \mapsto t$. We define $\mu_t$ to be the $k$-valued measure obtained by composing $\mu_{\tau}$ with this homomorphism. We say that the parameter $t$ is \defn{non-singular} if $\mu_t$ is regular, and \defn{singular} otherwise. If $k$ is a field of characteristic~0, then the formula in Proposition~\ref{prop:mu-formula} shows that $t$ is singular if and only it $t \in \bN$.

The measure $\mu$ is valued in the domain $R$ and never takes the value~0, and so for an embedding $i \colon T \to T'$ we have $\mu(i)=\mu(T) \mu(T')^{-1}$. The same formula is true for $\mu_t$ provided $\mu_t(T')$ is a unit. In general, to compute $\mu_t(i)$, first compute $\mu_{\tau}(i)$ and then make the substitution $\tau \mapsto t$. Practically speaking, this is just $\mu_t(T) \mu_t(T')^{-1}$, with the zero factors in the numerator and denominator formally canceled.

\subsection{Measures at finite level} \label{ss:finite-meas}

Let $n \ge 3$ be an integer, and let $\mu_n$ be the $\bZ[1/(n-1)]$-valued measure on $\fT$ at parameter $t=n$. It follows from Proposition~\ref{prop:mu-formula} that $\mu_n(T)$ vanishes if and only if $T$ has level $>n$. Thus by Proposition~\ref{prop:induced}, $\mu_n$ induces a $\bZ[1/(n-1)]$-valued measure $\mu_n'$ on $\fT_n$ that takes only non-zero values. We thus see that $\mu'_n$ is regular if regarded as a $\bQ$-valued measure. In fact, Proposition~\ref{prop:induced} shows that $\mu'_n$ is regular if regarded as a $\bZ[1/n!]$-valued measure; in particular, the mod $p$ reduction of $\mu'_n$ is regular for $p>n$.

\begin{remark}
One can compute $\Theta(\fT_n)$ using the same approach we used for $\Theta(\fT)$; it is isomorphic to
\begin{displaymath}
\bZ[u]/((n-1)u^2+(n-2)u)
\end{displaymath}
In particular, the spectrum of $\Theta(\fT_n) \otimes \bQ$ consists of two points, one of which corresponds to $\mu'_n$. The details will appear in \cite{Nekrasov}.
\end{remark}

\subsection{The case $t=\infty$} \label{ss:inf-meas}

Recall that we defined the measure $\mu$ using the ring homomorphism
\begin{displaymath}
\bZ[u,v]/(uv) \to R, \qquad u \mapsto -\frac{\tau-2}{\tau-1}, \qquad v \mapsto 0.
\end{displaymath}
We can specialize this homomorphism to $\tau=\infty$; that is, we have a well-defined homomorphism
\begin{displaymath}
\bZ[u,v]/(uv) \to R, \qquad u \mapsto -1, \qquad v \mapsto 0.
\end{displaymath}
We let $\mu_{\infty}$ be the resulting $\bZ$-valued measure for $\fT$. It satisfies
\begin{displaymath}
\mu_{\infty}(x_1)=\mu_{\infty}(x_m)=1, \qquad \mu_{\infty}(x_2)=0, \qquad \mu_{\infty}(x_3)=\mu_{\infty}(y)=\mu_{\infty}(z) = -1
\end{displaymath}
where here $m \ge 4$. 
 In particular we see that $\mu_{\infty}$ is not regular, and does not give rise to semi-simple pre-Tannakian category.  As with the other singular parameter values, it is interesting to then ask what is its semisimplification, and whether or not it admits an abelian envelope, but we won't pursue those any further here.

\section{Tensor categories associated to Fra\"iss\'e classes} \label{s:tencat}

\subsection{Overview}

In this section, we recall the general construction of the tensor category $\uPerm_k(\fF; \mu)$, and establish some new results about it. We give a brief overview. In \S \ref{ss:tensor}, we recall terminology related to tensor categories. In \S \ref{ss:uperm}, we give the construction of $\uPerm_k(\fF; \mu)$. In fact, while we attribute this construction to \cite{repst}, it does not actually appear there in this form; that paper uses the (more or less equivalent) language of oligomorphic groups. We explain this perspective in \S \ref{ss:oligo}. In \S \ref{ss:ss-crit}, we give a criterion for the Karoubi envelope of $\uPerm_k(\fF; \mu)$ to be semi-simple, and in \S \ref{ss:special} we determine the semi-simplification of $\uPerm_k(\fF; \mu)$. These results will be used to establish our main results on arboreal tensor categories.

\subsection{Background on tensor categories} \label{ss:tensor}

Let $k$ be a commutative ring. A \defn{tensor category} is a $k$-linear category $\cC$ equipped with a symmetric monoidal structure that is $k$-bilinear. Let $\bbone$ be the monoidal unit of $\cC$. A \defn{dual} of an object $X$ is an object $Y$ together with maps $X \otimes Y \to \bbone$ and $\bbone \to X \otimes Y$ satisfying the usual identities (see \cite[\S 2.10]{EGNO}). We say that $X$ is \defn{rigid} if it has a dual, and we say that $\cC$ is \defn{rigid} if all of its objects are. A \defn{tensor functor} is a $k$-linear symmetric monoidal functor of tensor categories.

Suppose now that $k$ is a field. We say that the tensor category $\cC$ is \defn{pre-Tannakian} if the following condition holds:
\begin{enumerate}
\item $\cC$ is abelian and all objects have finite length.
\item All $\Hom$ spaces in $\cC$ are finite dimensional over $k$.
\item $\cC$ is rigid.
\item $\End(\bbone)=k$.
\end{enumerate}
We warn the reader that our usage of ``tensor category'' does not match that of \cite{EGNO}.

\subsection{The main construction} \label{ss:uperm}

Let $\fF$ be a Fra\"iss\'e class and let $\mu$ be a measure for $\fF$ valued in a commutative ring $k$. We construct a rigid $k$-linear tensor category $\uPerm_k(\fF; \mu)$. This construction is equivalent to the one in \cite{repst}, as we explain in \S \ref{ss:oligo} below.

We first construct a $k$-linear category $\uPerm^0_k(\fF; \mu)$. The objects of this category are the members of $\fF$; for a structure $X$ in $\fF$, we write $\Vec_X$ for the corresponding object. Write $\Amalg(X,Y)$ for the set of isomorphism classes of amalgamations of $X$ and $Y$. We put
\begin{displaymath}
\Hom(\Vec_X, \Vec_Y) = k[\Amalg(X,Y)],
\end{displaymath}
where $k[S]$ indicates the free $k$-module with basis $S$. For an amalgamation $(Z,i,j)$ of $X$ and $Y$, we typically write $\phi_Z$ for the map $\Vec_X \to \Vec_Y$, omitting $i$ and $j$ from the notation. We now explain composition. Suppose we map morphisms
\begin{displaymath}
\phi_{Y_1} \colon \Vec_{X_1} \to \Vec_{X_2}, \qquad \phi_{Y_2} \colon \Vec_{X_2} \to \Vec_{X_3}.
\end{displaymath}
We define
\begin{displaymath}
\phi_{Y_2} \circ \phi_{Y_1} = \sum_{Y_3} \sum_Z \mu(Y_3 \to Z) \cdot \phi_{Y_3},
\end{displaymath}
where here $Y_3$ varies over all amalgamations of $X_1$ and $X_3$, and $Z$ varies over all amalgamations of $X_1$, $X_2$, and $X_3$ extending $Y_1$, $Y_2$, and $Y_3$. This means that $Z$ is a structure equipped with jointly surjective embeddings $X_i \to Z$, for $1 \le i \le 3$, such that the joint images of $X_1$ and $X_2$ is isomorphic to $Y_1$ (as an amalgamation of $X_1$ and $X_2$), and similarly for $Y_2$ and $Y_3$. Note that $Y_3$ is determined from $Z$, so one can rewrite the above formula as a single sum. See Example~\ref{ex:arboreal} for a concrete example of composition. We extend the composition law $k$-bilinearly to arbitrary morphisms. One can verify the category axioms directly, though this also follows from the discussion in \S \ref{ss:oligo} and the results of \cite{repst}. We note that the identity morphism of $\Vec_X$ is $\phi_X$, where $X$ is regarded as an amalgamation of $X$ and $X$ in the natural manner.

Suppose $i \colon X \to Y$ is an embedding of structures. Then $(Y, i, \id_Y)$ is an amalgamation of $X$ and $Y$. We write $\beta_i \colon \Vec_X \to \Vec_Y$ and $\alpha_i \colon \Vec_Y \to \Vec_X$ for the corresponding morphisms. These morphisms generate all morphisms. Indeed, suppose $\phi_Z \colon \Vec_X \to \Vec_Y$ is a morphism, where $(Z,i,j)$ is an amalgamation of $X$ and $Y$. Then one can verify that $\phi_Z=\alpha_j \circ \beta_i$.

We now define $\uPerm_k(\fF; \mu)$ to be the additive envelope of $\uPerm^0_k(\fF; \mu)$. This category carries a tensor product $\uotimes$ given on objects by
\begin{displaymath}
\Vec_X \uotimes \Vec_Y = \bigoplus_{Z \in \Amalg(X,Y)} \Vec_Z.
\end{displaymath}
We omit the discussion of how $\uotimes$ acts on morphisms, and the various other aspects of the tensor structure. This all follows from the work in \cite{repst}, as explained in \S \ref{ss:oligo}.

The categorical dimension of $\Vec_X$ is $\mu(X)$. The categorical trace of the endomorphism $\phi_Z$ of $\Vec_X$ is $\mu(X)$ if $\phi_Z$ is the identity, and~0 otherwise.

\subsection{The oligomorphic group perspective} \label{ss:oligo}

Suppose $\fF$ is a Fra\"iss\'e class. Then there exists a structure $\Omega$ with the following properties:
\begin{enumerate}
\item $\Omega$ is countable.
\item The \defn{age} of $\Omega$, that is, the class of finite structures embedding into $\Omega$, is $\fF$.
\item $\Omega$ is \defn{homogeneous}: if $X$ and $Y$ are finite substructures of $\Omega$ and $i \colon X \to Y$ is an isomorphism then there exists an automorphism $\sigma$ of $\Omega$ such that $\sigma \vert_X = i$.
\end{enumerate}
The structure $\Omega$ is called the \defn{Fra\"iss\'e limit} of $\Omega$, and is unique up to isomorphism. See \cite[\S 2.6]{CameronBook} or \cite[\S 6.2]{repst} for details.

Let $G$ be the automorphism group of $\Omega$. The $G$-orbits on the set $\Omega^{(n)}$ of $n$-element subsets of $\Omega$ are naturally in bijection with isomorphism classes of $n$-element structures in $\fF$; in particular, there are finitely many such orbits. It follows that $G$ has finitely many orbits on $\Omega^n$ for all $n \ge 0$; in other words, $G$ is an \defn{oligomorphic} permutation group.

For a finite subset $A$ of $\Omega$, let $G(A)$ be the subgroup of $G$ fixing each element of $A$. There is a topology on $G$ for which the $G(A)$'s constitute a neighborhood basis of the identity. With this topology, $G$ is a \defn{pro-oligomorphic group}\footnote{We used the term ``admissible group'' in older versions of the paper.} in the sense of \cite[\S 2.2]{repst}. Let $\sE$ be the set of open subgroups of $G$ of the form $G(A)$ for some finite subset $A$ of $\Omega$. Then $\sE$ is a \defn{stabilizer class} in the sense of \cite[\S 2.6]{repst}.

We say that an action of $G$ on a set is \defn{smooth} if every stabilizer is open, and $\sE$-smooth if every stabilizer belongs to $\sE$. We say that a $G$-set is \defn{finitary} if it has finitely many orbits. If $X$ is a structure in $\fF$ then the set $\cX = \Omega^{[X]}$ of embeddings $X \to \Omega$ is a transitive $\sE$-smooth $G$-set, and all transitive $\sE$-smooth $G$-sets are of this form. If $i \colon X \to Y$ is an embedding of structures then there is a corresponding $G$-morphism $i^* \colon \cY \to \cX$, where $\cY=\Omega^{[Y]}$. Suppose that we have embeddings $i \colon X_0 \to X_1$ and $j \colon X_0 \to X_2$, and let $Y_1, \ldots, Y_n$ be the various amalgamations. Given embeddings $\alpha \colon X_1 \to \Omega$ and $\beta \colon X_2 \to \Omega$ that agree on $X_0$, the induced structure on $\im(\alpha) \cup \im(\beta)$ is an amalgamation of $X_1$ and $X_2$ over $X_0$, and thus one of the $Y_i$'s. It follows that
\begin{displaymath}
\cX_1 \times_{\cX_0} \cX_2 = \coprod_{i=1}^n \cY_i
\end{displaymath}
where the fiber product on the left is taken with respect to $i^*$ and $j^*$.

A $k$-valued \defn{measure} for $G$, relative to $\sE$, is a rule $\mu$ assigning to each map $f \colon \cY \to \cX$ of $\sE$-smooth $G$-sets, with $\cX$ finitary and $\cY$ transitive, a quantity $\mu(f)$ in $k$ satisfying certain axioms; see \cite[\S 4.5]{repst}. Suppose we have such a measure. By (a relative form of) \cite[\S 8.2]{repst}, there is a tensor category $\uPerm_k(G, \sE; \mu')$. The objects of this category are $\Vec_{\cX}$, where $\cX$ is a finitary $\sE$-smooth $G$-set. Morphisms $\Vec_{\cX} \to \Vec_{\cY}$ are given by $G$-invariant $\cY \times \cX$ matrices, and composition is given by matrix multiplication, as discussed in \cite[\S 7]{repst}. The tensor product $\uotimes$ is given on objects by $\Vec_{\cX} \uotimes \Vec_{\cY}=\Vec_{\cX \times \cY}$.

Suppose that $\mu$ is a measure for $\fF$ valued in a ring $k$. Then there is a unique measure $\mu'$ for $G$ relative to $\sE$ such that if $i \colon X \to Y$ is an embedding of structures then $\mu'(i^*)=\mu(i)$. This assertion is the content of \cite[Theorem~6.9]{repst}. We have an equivalence of categories
\begin{displaymath}
\uPerm_k(\fF; \mu) \cong \uPerm_k(G, \sE; \mu')
\end{displaymath}
If $X$ is a structure in $\fF$ and $\cX=\Omega^{[X]}$ is the corresponding $G$-set then $\Vec_X$ and $\Vec_{\cX}$ correspond under the above equivalence. Suppose $Y$ is a second structure and $\cY=\Omega^{[Y]}$. Morphisms $\Vec_X \to \Vec_Y$ are linear combinations of amalgamations of $X$ and $Y$, while morphisms $\Vec_{\cX} \to \Vec_{\cY}$ are $G$-invariant $\cY \times \cX$ matrices. Such a matrix is a function on $G \backslash (\cY \times \cX) \to k$. As we saw above, the $G$-orbits on $\cY \times \cX$ naturally correspond to amalgamations of $X$ and $Y$. Thus we have a canonical identification
\begin{displaymath}
\Hom(\Vec_X, \Vec_Y) = \Hom(\Vec_{\cX}, \Vec_{\cY}).
\end{displaymath}
We leave the remaining details of the equivalence to the reader.

\subsection{Semi-simplicity} \label{ss:ss-crit}

Let $N \ge 1$ be an integer and let $\mu$ be a regular measure for $\fF$ valued in $\bZ[1/N]$. Let $k$ be a field in which $N$ is invertible, and regard $\mu$ as $k$-valued via the natural map $\bZ[1/N] \to k$. We consider the following two conditions:
\begin{itemize}
\item[(SS1)] For every object $X$ of $\fF$, every prime divisor of $\# \Aut(X)$ is a factor of $N$. 
\item[(SS2)] Given a structure $Y$ in $\fF$, a proper substructure $X$, and an integer $h \ge 1$, there exists an embedding $j \colon X \to Z$ in $\fF$ that admits at least $h$ distinct extensions to $Y$.
\end{itemize}
The following proposition is a useful semi-simplicity criterion.

\begin{proposition} \label{prop:ss-crit}
Suppose (SS1) and (SS2) hold. Then the Karoubi envelope of $\uPerm_k(\fF; \mu)$ is a semi-simple pre-Tannakian category.
\end{proposition}

We follow \cite[\S 18.4]{repst}. Let $\Omega$ be the Fra\"iss\'e limit of $\fF$ and let $G$ be its automorphism group. We regard $\mu$ as a measure for $G$, and prove that the Karoubi envelope of $\uPerm_k(G; \mu)$ is semi-simple; this implies the proposition by the discussion in \S \ref{ss:oligo}. For a finite subset $A$ of $\Omega$, let $G(A)$ be the subgroup of $G$ fixing each element of $A$.

\begin{lemma} \label{lem:ss-crit-1}
Let $A$ and $B$ be finite subsets of $\Omega$ such that $G(B)$ is contained in $G(A)$ with finite index. Then $A=B$.
\end{lemma}

\begin{proof}
Replacing $B$ with $A \cup B$, we can assume $A \subset B$. We assume this containment is proper and reach a contradiction. Recall that $\Omega^{[A]}$ denotes the set of embeddings $A \to \Omega$, and that this is isomorphic to $G/G(A)$. There is a restriction map $\Omega^{[B]} \to \Omega^{[A]}$, which is $h$-to-1 where $h$ is the (finite) index of $G(B)$ in $G(A)$. This means that an embedding $A \to \Omega$ admits exactly $h$ extensions to $B$, which contradicts (SS2). To be a bit more careful, by (SS2) there is some embedding $A \to Z$ that admits $>h$ extensions to $B$. Since $Z$ embeds into $\Omega$, this gives an embedding $A \to \Omega$ with $>h$ extensions to $B$, a contradiction.
\end{proof}

For a finite subset $A$ of $\Omega$, let $G[A]$ denote the subgroup of $G$ fixing $A$ as a set. The group $G[A]$ normalizes $G(A)$, and $G[A]/G(A)$ is naturally isomorphic to $\Aut(A)$. By the above lemma, $G[A]$ is the full normalizer of $G(A)$: indeed, if $g$ normalizes $G(A)$ then $G(A)=gG(A)g^{-1}=G(gA)$, and so $A=gA$ by the lemma, i.e., $g \in G[A]$. We let $\sE^+$ be the set of all subgroups $U$ of $G$ such that $G(A) \subset U \subset G[A]$ for some finite subset $A$ of $\Omega$.

\begin{lemma} \label{lem:ss-crit-2}
We have the following:
\begin{enumerate}
\item If $V \subset U$ is a finite index containment with $V \in \sE^+$ then $U \in \sE^+$.
\item If $X$ is an $\sE^+$-smooth $G$-set and $H$ is a subgroup of $\Aut(X)$ then $X/H$ is $\sE^+$-smooth.
\item If $X$ is an $\sE^+$-smooth $G$-set then so is $X^{(n)}$ for any $n \ge 1$.
\end{enumerate}
\end{lemma}

\begin{proof}
(a) Let $A$ be a finite subset of $\Omega$ with $G(A) \subset V \subset G[A]$. Since $V$ normalizes $G(A)$ and has finite index in $U$, it follows that there are finitely many conjugates of $G(A)$ in $U$, each of which has finite index in $U$. The intersection of these conjugates has the form $G(B)$ and is contained in $G(A)$ with finite index, and is therefore equal to $G(A)$ by Lemma~\ref{lem:ss-crit-1}. Thus $G(A)$ is normal in $U$, and so $U$ is contained in the normalizer $G[A]$ of $G(A)$. Hence $U \in \sE^+$, as required.

(b) It suffices to treat the case where $X$ is transitive, say $X=G/V$ with $V \in \sE^+$. Then $X/H=G/U$, where $V \subset U$ with finite index. Thus $U \in \sE^+$ by (a), and so $X$ is $\sE^+$-smooth.

(c) This follows from (b), as $X^{[n]}$ is $\sE^+$-smooth and $X^{(n)}=X^{[n]}/\fS_n$.
\end{proof}

\begin{lemma} \label{lem:ss-crit-3}
The $\bZ[1/N]$-valued measure $\mu$ for $(G,\sE)$ extends uniquely to a $\bZ[1/N]$-valued measure $\mu^+$ for $(G,\sE^+)$, which is still regular.
\end{lemma}

\begin{proof}
Let $V \in \sE^+$. By definition, there is a finite subset $A$ of $\Omega$ such that $G(A) \subset V \subset G[A]$. Since these containments are finite index, we must have
\begin{displaymath}
\mu^+(G/V) = [V:G(A)]^{-1} \cdot \mu(G/G(A)).
\end{displaymath}
by the same reasoning as in \cite[Corollary~4.5]{repst}, assuming $\mu^+$ is a measure. In fact, if we define $\mu^+$ by the above formula, then $\mu^+$ is a well-defined $\bQ$-valued measure, as asserted by \cite[\S 4.2(d)]{repst}. By (SS1), the prime factors of $\# \Aut(A) = [G[A]:G(A)]$ divide $N$, and so the same is true for the prime factors of $[V:G(A)]$. Thus $\mu^+(G/V)$ is a unit of the ring $\bZ[1/N]$, and so $\mu^+$ is regular.
\end{proof}

\begin{proof}[Proof of Proposition~\ref{prop:ss-crit}]
By \cite[Theorem~7.18]{repst}, nilpotent endomorphisms in the category $\uPerm_k(G, \sE^+; \mu^+)$ have trace~0. This theorem is not formulated in the relative case, but it holds there assuming the stabilizer class satisfies condition $(\ast)$ of \cite[Remark~5.5]{repst}, which $\sE^+$ does by Lemma~\ref{lem:ss-crit-2}(c). The key point here is that for this theorem we need to make use of the higher traces introduced in \cite[\S 7.4]{repst}, which are defined provided $(\ast)$ holds. See \cite[\S 7.8]{repst} for further discussion.

Since $\mu^+$ is regular, \cite[Proposition~7.23]{repst} shows that the trace pairing on endomorphism algebras in $\uPerm_k(G, \sE^+; \mu^+)$ is non-degenerate. Combined with the previous paragraph, this implies that endomorphism algebras are semi-simple; see \cite[Proposition~7.24]{repst}. Since $\uPerm_k(G, \sE; \mu)$ is a full subcategory of $\uPerm_k(G, \sE^+; \mu^+)$, its endomorphism algebras are semi-simple as well. Thus its Karoubi envelope is semi-simple (see \cite[Proposition~2.4]{dblexp}), and so clearly pre-Tannakian as well.
\end{proof}

\subsection{Specialization} \label{ss:special}

Let $\fF$ be a Fra\"iss\'e class and let $\mu$ be a measure for $\fF$ valued in a domain $k$. Let $\fF'$ be the subclass consisting of structures with non-zero measure. As we have seen (Proposition~\ref{prop:induced}), $\fF'$ is itself a Fra\"iss\'e class and $\mu$ restricts to a nowhere zero measure $\mu'$ on $\fF'$. The following proposition connects the tensor categories associated to these two Fra\"iss\'e classes.

\begin{proposition} \label{prop:special}
There is a tensor functor
\begin{displaymath}
\Phi \colon \uPerm_k(\fF; \mu) \to \uPerm_k(\fF'; \mu')
\end{displaymath}
satisfying
\begin{displaymath}
\Phi(\Vec_X) = \begin{cases}
\Vec_X & \text{if $X \in \fF'$} \\
0 & \text{otherwise} \end{cases}
\qquad
\Phi(\phi_Z) = \begin{cases}
\phi_Z & \text{if $Z \in \fF'$} \\
0 & \text{otherwise} \end{cases}
\end{displaymath}
The functor $\Phi$ is essentially surjective and full.
\end{proposition}

\begin{proof}
Ignoring the tensor structure, we have specified $\Phi$ on objects and morphisms. We first verify that this yields a well-defined functor. Note that if $X \to Y$ is an embedding and $Y \in \fF'$ then $X \in \fF'$. We thus see that if $\phi_Z \colon \Vec_X \to \Vec_Y$ is a morphism and $Z \in \fF'$ then both $X$ and $Y$ are in $\fF'$, as they embed into $Z$, and so $\phi_Z=\Phi(\phi_Z)$ is indeed a morphism from $\Phi(\Vec_X)=\Vec_X$ to $\Phi(\Vec_Y)$. Of course, if $Z \not\in \fF'$ then $\Phi(\phi_Z)=0$ is a morphism between $\Phi(\Vec_X)$ and $\Phi(\Vec_Y)$ whatever these objects are. Thus $\Phi$ induces a map on $\Hom$ spaces in the requisite manner. It is also clear that $\Phi$ preserves identity morphisms.

We now verify that $\Phi$ is compatible with composition. Suppose we have morphisms
\begin{displaymath}
\phi_{Y_1} \colon \Vec_{X_1} \to \Vec_{X_2}, \qquad \phi_{Y_2} \colon \Vec_{X_2} \to \Vec_{X_3}
\end{displaymath}
in $\uPerm_k(\fF; \mu)$. Recall that
\begin{displaymath}
\phi_{Y_2} \circ \phi_{Y_1} = \sum_{Y_3} \sum_Z \mu(Y_3 \to Z) \cdot \phi_{Y_3},
\end{displaymath}
where $Y_3$ varies over all amalgamations (in $\fF$) of $X_1$ and $X_3$, and $Z$ varies over all amalgamations (in $\fF$) of $X_1$, $X_2$, and $X_3$ extending $Y_1$, $Y_2$, and $Y_3$. We claim that $\Phi(\phi_{Y_2} \circ \phi_{Y_1})$ is given by the same formula, except with $Y_3$ and $Z$ varying over amalgamations in $\fF'$. Indeed, in the original sum, any $\phi_{Y_3}$ term with $Y_3 \not \in \fF'$ is killed by $\Phi$; thus we can restrict to $Y_3 \in \fF'$. Next, observe that if $Y_3 \in \fF'$ and $Z \not\in \fF'$ then $\mu(Y_3 \to Z)=\mu(Z) \mu(Y_3)^{-1}=0$. Thus only those $Z \in \fF'$ have non-zero contribution in the sum, which is why we can restrict to $Z \in \fF'$. This establishes the claim. We thus see that $\Phi(\phi_{Y_2} \circ \phi_{Y_1})$ is exactly the composition of $\phi_{Y_2} \circ \phi_{Y_1}$ as computed in $\uPerm_k(\fF'; \mu)$. This shows that $\Phi$ is a well-defined functor.

We have a natural isomorphism
\begin{displaymath}
\Phi(\Vec_X \uotimes \Vec_Y) = \Phi(\Vec_X) \uotimes \Phi(\Vec_Y).
\end{displaymath}
Indeed, if $X$ and $Y$ both belong to $\fF'$ then this follows immediately from the definition, and otherwise both sides vanish (note that in this case all amalgamations do not belong to $\fF'$). We omit the rest of the verification that $\Phi$ is a tensor functor.

Finally, it is clear that $\Phi$ is essentially surjective. If $\phi_Z \colon \Vec_X \to \Vec_Y$ is a morphism in $\uPerm_k(\fF'; \mu)$ then it is also a morphism in $\uPerm_k(\fF; \mu)$, and $\Phi(\phi_Z)=\phi_Z$. This shows that $\Phi$ is full.
\end{proof}

Let $k$ be a field, and let $\cC$ be a $k$-linear tensor category. Under certain conditions, one can define the ideal of negligible morphisms in $\cC$ and prove that the quotient category is semi-simple; this category is called the \defn{semi-simplification} of $\cC$. See \cite[\S 2.3]{EtingofOstrik} for details.

\begin{corollary} \label{cor:special}
Suppose that $k$ is a field and $\uPerm_k(\fF'; \mu')^{\rm kar}$ is semi-simple. Then nilpotent endomorphisms in $\uPerm_k(\fF; \mu)^{\rm kar}$ have categorical trace~0, and the semi-simplification of $\uPerm_k(\fF; \mu)^{\rm kar}$ is $\uPerm_k(\fF'; \mu')^{\rm kar}$.
\end{corollary}

\begin{proof}
Put $\cC=\uPerm_k(\fF; \mu)^{\rm kar}$ and $\cD=\uPerm_k(\fF'; \mu')^{\rm kar}$. By Proposition~\ref{prop:special} there is a tensor functor $\Phi \colon \cC \to \cD$. Since the $\cD$ is semi-simple, any nilpotent endomorphism in it has categorical trace~0. It follows that the same is true for $\cC$. Indeed, if $\phi$ is a nilpotent endomorphism in $\cC$ then $\Phi(\phi)$ is a nilpotent endomorphism in $\cD$, and thus of trace~0; since $\Phi$ is a tensor functor, it preserves traces, and so $\phi$ has trace~0. We thus see that the semi-simplification of $\cC$ is defined. Since $\Phi$ is full and essentially surjective, it is necessarily the case that $\Phi$ is the semi-simplification of $\cC$.
\end{proof}

\section{Tensor categories associated to trees} \label{s:arboreal}

\subsection{Overview}

In this section, we construct the arboreal tensor categories $\cC_k(t)$ and $\cD_k(n)$ and prove our main theorems about them. The categories are defined in \S \ref{ss:treecat}. In \S \ref{ss:finite}, we prove our main result about the finite level category $\cD_k(n)$. In \S \ref{ss:singular} and \S \ref{ss:nonsing}, we prove our main results about the infinite level category $\cC_k(t)$ in the case of singular and non-singular parameters, respectively. Finally, in \S \ref{ss:further}, we make some additional remarks about arboreal tensor categories.

\subsection{Arboreal categories and algebras} \label{ss:treecat}

Recall (\S \ref{ss:goodmeas}) that if $t$ is an element of a ring $k$ such that $t-1$ is a unit then we have a measure $\mu_t$ for $\fT$ valued in $k$. Similarly (\S \ref{ss:finite-meas}), if $n \ge 3$ is an integer such that $n-1$ is a unit in $k$ then we have a measure $\mu'_n$ for $\fT_n$ valued in $k$. We define
\begin{displaymath}
\cC_k(t) = \uPerm_k(\fT; \mu_t)^{\rm kar}, \qquad
\cD_k(n) = \uPerm_k(\fT_n; \mu'_n)^{\rm kar},
\end{displaymath}
where the superscript kar denotes the Karoubi envelope. These rigid tensor categories, which we call \defn{arboreal tensor categories}, are the main objects of study in this section.

Let $\bT=(T_1, \ldots, T_r)$ be a tuple of trees. We let $\Vec_{\bT}$ be the direct sum $\Vec_{T_1} \oplus \cdots \oplus \Vec_{T_r}$, which we regard as an object of $\cC_k(t)$ or $\cD_k(n)$ (in the latter case, we simply omit summands of level $>n$). We define the \defn{arboreal algebras} $A^{\bT}_k(t)$ and $B^{\bT}_k(n)$ to be the endomorphism algebras of $\Vec_{\bT}$ in $\cC_k(t)$ and $\cD_k(n)$. The algebra $A^{\bT}_k(n)$ is similar in spirit to Brauer algebras, partition algebras, and other such diagram algebras.

\begin{example} \label{ex:arboreal}
We make a few comments to illustrate the nature of arboreal algebras. Let $T$ be the following tree
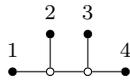
\begin{displaymath}
\tikz{
\node[boron] (A) at (0,0) {};
\node[boron] (B) at (0.5,0) {};
\node[leaf,label={\tiny 1}] (x) at (-0.5,0) {};
\node[leaf,label={\tiny 4}] (y) at (1,0) {};
\node[leaf,label={\tiny 2}] (a) at (0,0.5) {};
\node[leaf,label={\tiny 3}] (b) at (0.5,0.5) {};
\path[draw] (A)--(B);
\path[draw] (A)--(x);
\path[draw] (B)--(y);
\path[draw] (A)--(a);
\path[draw] (B)--(b);
}
\end{displaymath}
We consider the algebra $A^T_k(t)$. This algebra has a basis consisting of the self-amalgamations of $T$. To work with such amalgamations, it is convenient to introduce a copy $T'$ of $T$ with vertices labeled $1', \ldots, 4'$. A self-amalgamation is then a tree with vertices labeled by $1, 1', \ldots, 4, 4'$ such that the subtree on $1, \ldots, 4$ is $T$, and the subtree on $1', \dots, 4'$ is $T'$. Note that a single vertex can have two labels (one primed, one not). The following trees are examples such amalgamations
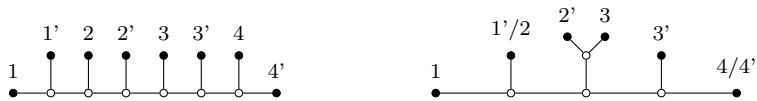
\begin{displaymath}
\tikz[baseline=0pt]{
\node[boron] (A) at (0,0) {};
\node[boron] (B) at (0.5,0) {};
\node[boron] (C) at (1,0) {};
\node[boron] (D) at (1.5,0) {};
\node[boron] (E) at (2,0) {};
\node[boron] (F) at (2.5,0) {};
\node[leaf,label={\tiny 1}] (a1) at (-0.5,0) {};
\node[leaf,label={\tiny 1'}] (a2) at (0,.5) {};
\node[leaf,label={\tiny 2}] (b) at (0.5,0.5) {};
\node[leaf,label={\tiny 2'}] (c) at (1,0.5) {};
\node[leaf,label={\tiny 3}] (d) at (1.5,0.5) {};
\node[leaf,label={\tiny 3'}] (e) at (2,0.5) {};
\node[leaf,label={\tiny 4}] (f1) at (2.5,0.5) {};
\node[leaf,label={\tiny 4'}] (f2) at (3,0) {};
\path[draw] (A)--(B);
\path[draw] (B)--(C);
\path[draw] (C)--(D);
\path[draw] (D)--(E);
\path[draw] (E)--(F);
\path[draw] (A)--(a1);
\path[draw] (A)--(a2);
\path[draw] (B)--(b);
\path[draw] (C)--(c);
\path[draw] (D)--(d);
\path[draw] (E)--(e);
\path[draw] (F)--(f1);
\path[draw] (F)--(f2);
}
\qquad\qquad
\tikz[baseline=0pt]{
\node[boron] (A) at (0,0) {};
\node[boron] (B) at (1,0) {};
\node[boron] (C) at (1, .5) {};
\node[boron] (D) at (2, 0) {};
\node[leaf,label={\tiny 1}] (x) at (-1,0) {};
\node[leaf,label={\tiny 1'/2}] (a) at (0,0.5) {};
\node[leaf,label={\tiny 2'}] (c1) at (0.75,0.75) {};
\node[leaf,label={\tiny 3}] (c2) at (1.25,.75) {};
\node[leaf,label={\tiny 3'}] (d) at (2,.5) {};
\node[leaf,label={\tiny 4/4'}] (y) at (3,0) {};
\path[draw] (A)--(B);
\path[draw] (B)--(C);
\path[draw] (B)--(D);
\path[draw] (A)--(x);
\path[draw] (A)--(a);
\path[draw] (C)--(c1);
\path[draw] (C)--(c2);
\path[draw] (D)--(d);
\path[draw] (D)--(y);
}
\end{displaymath}
Call the above two trees $X$ and $Y$. Let us explain how to compute their product in $A^T_k(t)$. First, let $Y'$ be a copy of $Y$ where $1, \ldots, 4$ are renamed $1', \ldots, 4'$, and $1', \ldots, 4'$ are renamed $1'', \ldots, 4''$. Let $\Sigma$ be the set of trees with labels $1, 1', 1'', \ldots, 4, 4', 4''$ such that the subtree on $1, 1', \ldots, 4, 4'$ is $X$ and the subtree on $1', 1'', \ldots, 4', 4''$ is $Y'$. For $Z \in \Sigma$, let $\ol{Z}$ be the subtree on $1, 1'', \ldots, 4, 4''$, with double primes converted to single primes. Then
\begin{displaymath}
X \cdot Y = \sum_Z \mu(\ol{Z} \to Z) \cdot \ol{Z},
\end{displaymath}
where here the inclusion of $\ol{Z}$ into $Z$ maps the single prime vertices of $\ol{Z}$ back to the corresponding double prime vertices. For example, the following tree $Z$ is an element of $\Sigma$
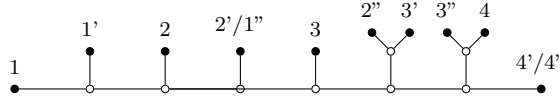
\begin{displaymath}
\tikz[baseline=0pt]{
\node[boron] (A) at (0,0) {};
\node[boron] (B) at (1,0) {};
\node[boron] (C) at (2,0) {};
\node[boron] (D) at (3,0) {};
\node[boron] (E) at (4,0) {};
\node[boron] (F) at (4,.5) {};
\node[boron] (G) at (5,0) {};
\node[boron] (H) at (5,.5) {};
\node[leaf,label={\tiny 1}] (a1) at (-1,0) {};
\node[leaf,label={\tiny 1'}] (a2) at (0,0.5) {};
\node[leaf,label={\tiny 2}] (b) at (1,.5) {};
\node[leaf,label={\tiny 2'/1''}] (c) at (2,.5) {};
\node[leaf,label={\tiny 3}] (d) at (3,.5) {};
\node[leaf,label={\tiny 2''}] (f1) at (3.75,.75) {};
\node[leaf,label={\tiny 3'}] (f2) at (4.25,.75) {};
\node[leaf,label={\tiny 3''}] (h1) at (4.75,.75) {};
\node[leaf,label={\tiny 4}] (h2) at (5.25,.75) {};
\node[leaf,label={\tiny 4'/4''}] (g) at (6,0) {};
\path[draw] (A)--(B);
\path[draw] (B)--(C);
\path[draw] (B)--(D);
\path[draw] (D)--(E);
\path[draw] (E)--(F);
\path[draw] (E)--(G);
\path[draw] (G)--(H);
\path[draw] (A)--(a1);
\path[draw] (A)--(a2);
\path[draw] (B)--(b);
\path[draw] (C)--(c);
\path[draw] (D)--(d);
\path[draw] (F)--(f1);
\path[draw] (F)--(f2);
\path[draw] (H)--(h1);
\path[draw] (H)--(h2);
\path[draw] (G)--(g);
}
\end{displaymath}
The tree $\ol{Z}$ in this case is
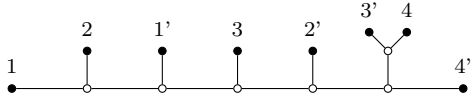
\begin{displaymath}
\tikz[baseline=0pt]{
\node[boron] (A) at (0,0) {};
\node[boron] (B) at (1,0) {};
\node[boron] (C) at (2,0) {};
\node[boron] (D) at (3,0) {};
\node[boron] (E) at (4,0) {};
\node[boron] (F) at (4,.5) {};
\node[leaf,label={\tiny 1}] (a1) at (-1,0) {};
\node[leaf,label={\tiny 2}] (a2) at (0,0.5) {};
\node[leaf,label={\tiny 1'}] (b) at (1,0.5) {};
\node[leaf,label={\tiny 3}] (c) at (2,.5) {};
\node[leaf,label={\tiny 2'}] (d) at (3,.5) {};
\node[leaf,label={\tiny 3'}] (f1) at (3.75,.75) {};
\node[leaf,label={\tiny 4}] (f2) at (4.25,.75) {};
\node[leaf,label={\tiny 4'}] (e) at (5,0) {};
\path[draw] (A)--(B);
\path[draw] (B)--(C);
\path[draw] (C)--(D);
\path[draw] (D)--(E);
\path[draw] (E)--(F);
\path[draw] (A)--(a1);
\path[draw] (A)--(a2);
\path[draw] (B)--(b);
\path[draw] (C)--(c);
\path[draw] (D)--(d);
\path[draw] (E)--(e);
\path[draw] (F)--(f1);
\path[draw] (F)--(f2);
}
\end{displaymath}
We have
\begin{displaymath}
\mu(Z) = \frac{t (t-2)^8}{(t-1)^{10}}, \qquad
\mu(\ol{Z}) = \frac{t (t-2)^6}{(t-1)^8}
\end{displaymath}
and so
\begin{displaymath}
\mu(\ol{Z} \to Z) = \frac{\mu(Z)}{\mu(\ol{Z})} = \frac{(t-2)^2}{(t-1)^2}.
\end{displaymath}
We thus see that the contribution of $Z$ to the product is the above scalar times the basis element $\ol{Z}$. To fully compute $X \cdot Y$, we would have to list all elements of $\Sigma$ and carry out the above computation for each one. This can get quite complicated even for small examples.
\end{example}

\subsection{Finite level} \label{ss:finite}

The following is our main result about the $\cD_k(n)$ categories. We note that the case $n=3$ is \cite[Theorem~18.10]{repst}.

\begin{theorem} \label{thm:finite}
Let $n \ge 3$ and let $k$ be a field in which $n!$ is invertible. Then $\cD_k(n)$ is a semi-simple pre-Tannakian category.
\end{theorem}

To prove the theorem, we apply Proposition~\ref{prop:ss-crit} with $N=n!$. We must verify the two conditions (SS1) and (SS2), which we do in the next two lemmas.

\begin{lemma}
If $T$ is a tree of level $\le n$ then every prime divisor of $\# \Aut(T)$ is $\le n$. In particular, $\fT_n$ satisfies (SS1) with $N=n!$.
\end{lemma}

\begin{proof}
A well-known result of P\'olya describes $\Aut(T)$ as a product of iterated wreath product of symmetric groups. The symmetric groups appearing in this description are among $\fS_2, \ldots, \fS_m$ where $m$ is the maximum valence of a vertex in $T$ (excluding the case where $T$ is a single edge). The result thus follows.
\end{proof}

\begin{lemma}
The class $\fT_n$ satisfies (SS2) from \S \ref{ss:ss-crit}.
\end{lemma}

\begin{proof}
Let $Y$ be a tree in $\fT_n$, let $X$ be a proper subtree, and let $m>0$ be given. We must produce an embedding $X \to Z$ in $\fT_n$ that extends to $Y$ in at least different $h$ ways. Pick a leaf $a$ of $Y$ that does not belong to $X$. Draw $Y$ as
\begin{displaymath}
\tikz{
\node[boron] (A) at (0,0) {};
\node[leaf,label={\tiny $a$}] (y) at (.5, 0) {};
\node[mid] (a1) at (-.5,-.1) {};
\node[mid] (a2) at (-.5,0) {};
\node[mid] (a3) at (-.5,.1) {};
\node[draw] () at (-.76, 0) {\tiny $R$};
\path[draw] (A)--(y);
\path[draw] (A)--(a1);
\path[draw] (A)--(a2);
\path[draw] (A)--(a3);
}
\end{displaymath}
where $R$ is the rest of the tree, which includes all of $X$. (Note that if $Y$ has only one or two leaves the above picture is inaccurate, but the following argument still goes through.) Define $Z$ to be the following tree
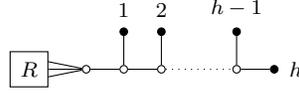
\begin{displaymath}
\tikz{
\node[boron] (A) at (0,0) {};
\node[boron] (B) at (.5, 0) {};
\node[boron] (C) at (1, 0) {};
\node[boron] (D) at (2, 0) {};
\node[leaf,label={\tiny $1$}] (b) at (.5, .5) {};
\node[leaf,label={\tiny $2$}] (c) at (1, .5) {};
\node[leaf,label={\tiny $h-1$}] (d1) at (2, .5) {};
\node[leaf,label=right:{\tiny $h$}] (d2) at (2.5, 0) {};
\node[mid] (a1) at (-.5,-.1) {};
\node[mid] (a2) at (-.5,0) {};
\node[mid] (a3) at (-.5,.1) {};
\node[draw] () at (-.76, 0) {\tiny $R$};
\path[draw] (A)--(a1);
\path[draw] (A)--(a2);
\path[draw] (A)--(a3);
\path[draw] (A)--(B);
\path[draw] (B)--(C);
\path[draw,dotted] (C)--(D);
\path[draw] (B)--(b);
\path[draw] (C)--(c);
\path[draw] (D)--(d1);
\path[draw] (D)--(d2);
}
\end{displaymath}
The tree $Z$ still has level $\le n$, since we have assumed $n \ge 3$. For $1 \le i \le h$, let $f_i \colon Y \to Z$ be the embedding that is the identity on $R$ and that maps $a$ to $i$. These all have the same restriction $f^{\circ}$ to $X$. Thus $f^{\circ} \colon X \to Z$ is an embedding that extends in $h$ different ways to $Y$, as required.
\end{proof}

\begin{corollary} \label{cor:Bss}
Let $k$ be a field in which $n!$ is invertible and let $\bT$ be a tuple of trees. Then the arboreal algebra $B^{\bT}_k(n)$ is semi-simple.
\end{corollary}

\subsection{Infinite level: singular parameters} \label{ss:singular}

Fix an integer $n \ge 3$. We now study the category $\cC_k(t)$ at parameter value $t=n$. We refer to such parameters as \defn{singular}.

\begin{theorem}
Let $k$ be a ring in which $n-1$ is invertible. Then there is a tensor functor
\begin{displaymath}
\Phi \colon \cC_k(n) \to \cD_k(n)
\end{displaymath}
given on objects by
\begin{displaymath}
\Phi(\Vec_T)=\begin{cases}
\Vec_T & \text{if $T$ has level $\le n$} \\
0 & \text{otherwise} \end{cases}
\end{displaymath}
The functor $\Phi$ is essentially surjective and full.
\end{theorem}

\begin{proof}
This follows from Proposition~\ref{prop:special} when $k=\bZ[1/(n-1)]$. The general case follows from this case via extension of scalars.
\end{proof}

The theorem has two important corollaries.

\begin{corollary} \label{cor:AB-isom}
Let $\bT=(T_1, \ldots, T_r)$ be a tuple of trees and let $k$ be a ring in which $n-1$ is invertible. There is a surjective ring homomorphism
\begin{displaymath}
F \colon A^{\bT}_k(n) \to B^{\bT}_k(n)
\end{displaymath}
If each $T_i$ has level $\le \tfrac{1}{2} n$ then this map is an isomorphism.
\end{corollary}

\begin{proof}
Since $\Phi$ is a full functor, it induces a surjective ring homomorphism $F$ as above. The algebra $A^{\bT}_k(n)$ has a basis consisting of elements $\phi_Z$, where $Z$ is an amalgamation of $T_i$ and $T_j$. The algebra $B^{\bT}_k(n)$ has a similar basis, except where $Z$ must have level $\le n$. The homomorphism $F$ maps $\phi_Z$ to $\phi_Z$ if $Z$ has level $\le n$, and~0 otherwise. Now, one easily sees that if $T$ and $T'$ are trees of levels $\ell_1$ and $\ell_2$ then any amalgamation has level $\le \ell_1+\ell_2$. It follows that if each $T_i$ has level $\le \tfrac{1}{2} n$ then any amalgamation of $T_i$ and $T_j$ has level $\le n$, and so $F$ is a bijection on basis vectors, and thus an isomorphism.
\end{proof}

\begin{corollary}
Let $k$ be a field in which $n!$ is invertible. Then any nilpotent endomorphism in $\cC_k(n)$ has categorical trace~0, and the semi-simplification of $\cC_k(n)$ is $\cD_k(n)$.
\end{corollary}

\begin{proof}
This follows from Corollary~\ref{cor:special} and Theorem~\ref{thm:finite}.
\end{proof}

\subsection{Infinite level: non-singular parameters} \label{ss:nonsing}

We now study the category $\cC_k(t)$ at non-singular values of the parameter $t$, i.e., values that are not natural numbers. The following is our main result:

\begin{theorem} \label{thm:nonsing}
Let $k$ be a field of characteristic~0 and let $t \in k \setminus \bN$ be a non-singular parameter. Then $\cC_k(t)$ is a semi-simple pre-Tannakian category.
\end{theorem}

We require a few lemmas before proving the theorem. The first is a weaker form of the Tits deformation theorem.

\begin{lemma} \label{lem:nonsing-1}
Let $R$ be an integral domain and let $A$ be a finite free $R$-algebra. Suppose there is a maximal ideal $\fm$ of $R$ such that $R/\fm$ is a perfect field and $A \otimes_R R/\fm$ is semi-simple. Then $A \otimes_R \Frac(R)$ is also semi-simple.
\end{lemma}

\begin{proof}
Since $A$ is finite free as an $R$-module, it admits a trace pairing. Let $\delta$ be its discriminant, i.e., the determinant of the matrix of the pairing. Since $A \otimes_R R/\fm$ is a semi-simple algebra over a perfect field, its trace pairing is non-degenerate. Thus $\delta$ is non-zero modulo $\fm$, and in particular a non-zero element of $R$. We thus see that the trace pairing on $A \otimes_R \Frac(R)$ is non-degenerate, and so this algebra is semi-simple.
\end{proof}

\begin{lemma} \label{lem:nonsing-2}
Let $k$ be an algebraic number field, let $t \in k \setminus \bN$, and let $r>0$ be an integer. Then there exists a prime ideal $\fp$ of $k$ over a rational prime number $p$ and an integer $n$ in the interval $(r,p)$ such that $t$ is integral at $\fp$ and $t \equiv n \pmod{\fp}$.
\end{lemma}

\begin{proof}
Let $\tilde{S}$ be the set of prime ideals of $k$ whose residue field is prime, i.e,. of the form $\bF_p$ for some prime number $p$. This set is infinite, and even of positive density by the Chebotarev density theorem. Let $S \subset \tilde{S}$ be the subset consisting of primes $\fp$ at which $t$ is $\fp$-integral. This condition excludes only finitely many primes, and so $S$ is still infinite. For a non-negative integer $a$, let $S_a \subset S$ be the set of primes $\fp$ such that $t \equiv a \pmod{\fp}$. This set is finite, since only finitely many primes divide the non-zero algebraic number $t-a$. (This is where we use $t \not\in \bN$.) Finally, let
\begin{displaymath}
S'=S \setminus (S_0 \cup S_1 \cup \cdots \cup S_r),
\end{displaymath}
which is infinite. Choose $\fp \in S'$, and let $p$ be the corresponding rational prime. Then $t$ is not congruent to any of $0, \ldots, r$ modulo $\fp$, but is congruent to some rational integer modulo $\fp$. We can thus find $n \in (r,p)$ such that $t$ is congruent to $n$ modulo $\fp$.
\end{proof}

\begin{lemma} \label{lem:nonsing-3}
Let $\bT=(T_1, \ldots, T_r)$ be a tuple of trees, let $k$ be a field of characteristic~0, and let $t \in k \setminus \bN$ be a non-singular parameter. Then the algebra $A^{\bT}_k(t)$ is semi-simple.
\end{lemma}

\begin{proof}
First suppose that $t$ is algebraic over $\bQ$, and let $k_0=\bQ(t)$ be the subfield of $k$ generated by $t$. Let $m$ be the maximum level of the $T_i$'s. By the previous lemma, there is a prime $\fp$ of $k_0$ over a rational prime $p$ and an integer $n$ in the interval $(2m,p)$ such that $t$ is integral at $\fp$ and $t \equiv n \pmod{\fp}$. Let $\cO \subset k_0$ be the ring of $\fp$-integral elements. We have isomorphisms
\begin{displaymath}
A^{\bT}_{\cO}(t) \otimes_{\cO} \cO/\fp \cong A^{\bT}_{\cO/\fp}(n) \cong B^{\bT}_{\cO/\fp}(n),
\end{displaymath}
where the first is simply the compatibility of $A$ with base change, and the second is Corollary~\ref{cor:AB-isom}. By Corollary~\ref{cor:Bss}, these rings are semi-simple. Thus, by Lemma~\ref{lem:nonsing-1}, we conclude that $A^{\bT}_{\cO}(t) \otimes_{\cO} k_0 = A^{\bT}_{k_0}(t)$ is semi-simple. Since $A^{\bT}_k(t)$ is simply an extension of scalars of this algebra, it too is semi-simple.

Now suppose that $t$ is transcendental over $\bQ$. Consider the algebra $A^{\bT}_{\bQ[\tau]}(\tau)$. By the previous paragraph, this is semi-simple modulo the maximal ideal $(\tau-t)$ for any non-singular algebra parameter $t$, such as $t=-1$. Thus by Lemma~\ref{lem:nonsing-1}, we find that $A^{\bT}_{\bQ(\tau)}(\tau)$ is semi-simple. The algebra $A^{\bT}_k(t)$ is a base change of this algebra along the field homomorphism $\bQ(\tau) \to k$ mapping $\tau$ to $t$, and is therefore semi-simple.
\end{proof}

\begin{proof}[Proof of Theorem~\ref{thm:nonsing}]
By Lemma~\ref{lem:nonsing-3}, the endomorphism algebra of any object in the category $\uPerm_k(\fT; \mu_t)$ is semi-simple. Thus its Karoubi envelope $\cC_k(t)$ is semi-simple. (See, e.g., \cite[Proposition~2.4]{dblexp} for this general fact.)
\end{proof}

\begin{remark} \label{rmk:ultra}
Let $I$ be an index set equipped with a non-principal ultrafilter. For each $i \in I$, let $k_i$ be a field and let $n_i$ be an integer with $n_i!$ invertible in $k_i$. Let $k$ be the ultraproduct of the $k_i$'s, and let $t$ be the element of $k$ represented by the tuple $(n_i)_{i \in I}$. One can show that $\cC_k(t)$ is equivalent to a full subcategory of the ultraproduct of the $\cD_{k_i}(n_i)$. It is possible to use this perspective to give a different proof of Theorem~\ref{thm:nonsing}.
\end{remark}

\begin{remark}
There is yet another approach to Theorem~\ref{thm:nonsing}. Let $\Omega$ be the Fra\"iss\'e limit of $\fT$, let $G$ be its automorphism group, and let $\sE$ be the stabilizer class as in \S \ref{ss:oligo}. The $R$-valued measure $\mu_{\tau}$ on $\fT$ corresponds to a measure for $G$ relative to $\sE$. Letting $\sE^+$ be as in \S \ref{ss:ss-crit}, one can show that $\mu$ extends to a measure for $G$ relative to $\sE^+$. This measure, it turns out, takes value in the ring of integer-valued polynomials with $\tau-1$ inverted. It follows that the measure satisfies Property~(P) from \cite[Definition~7.17]{repst}; see \cite[\S 15.5]{repst}. A slightly more general version of Proposition~\ref{prop:ss-crit} then implies that $\cC_k(t)$ is semi-simple. The proof that the ring of integer-valued polynomials satisfies Property~(P) relies on some number theory similar to that in Lemma~\ref{lem:nonsing-2}.
\end{remark}

\subsection{Further comments} \label{ss:further}

We make a few additional comments related to the material in this section.

(a) We have determined the measures on the Fra\"iss\'e class $\fT$ of trees. However, one might wonder what the measures are for the corresponding oligomorphic group $\Gamma$ (the automorphism group of the Fra\"iss\'e limit), which is perhaps a more natural question. This is addressed in \cite{Nekrasov}. We do not give the full answer here, but describe one of the key ideas. Define an ``enhanced tree'' to be a tree where each vertex is either marked as a leaf or a node. The leaves must have valence one, but there is no restriction on the valence of nodes; they are allowed to have valence one, two, or higher. There is a corresponding notion of ``enhanced tree structures.'' These structures form a Fra\"iss\'e class, and one can determine the measures through an analysis similar to that in \S \ref{s:trees}. This Fra\"iss\'e class is a closer approximation to $\Gamma$ than $\fT$, and once its measures are determined it is not difficult to go to $\Gamma$.

(b) Trees of level $ \le 2$ are not very interesting: there are only three, namely, those with at most two leaves. However, there are interesting ``enhanced trees'' of level $\le 2$, such as
\begin{displaymath}
\tikz{
\node[leaf] (x) at (-.5,0) {};
\node[boron] (A) at (0,0) {};
\node[boron] (B) at (.5, 0) {};
\node[boron] (C) at (1, 0) {};
\node[boron] (D) at (1.5, 0) {};
\node[boron] (E) at (2, 0) {};
\node[leaf] (y) at (2.5, 0) {};
\path[draw] (A)--(x);
\path[draw] (A)--(B);
\path[draw] (B)--(C);
\path[draw] (C)--(D);
\path[draw] (D)--(E);
\path[draw] (E)--(y);
}
\end{displaymath}
Such a tree has at most two vertices; by fixing one, the set of nodes inherits a total order. We thus see that this Fra\"iss\'e class is closely related to the Fra\"iss\'e class of totally ordered sets. If we redefine $\cD_k(n)$ using enhanced trees, then $\cD_k(2)$ is the $\bZ/2$-equivariantization of the Delannoy category \cite{line}, and $\cD_k(n)$ is unchanged for $n \ge 3$. We can similarly redefine $\cC_k(t)$, though this leads to the same category. We thus see that the semi-simplification of $\cC_k(2)$ is the $\bZ/2$-equivariantization of the Delannoy category.

(c) The measure $\mu_t$ on $\fT$ is defined for $t=0$ and $t=\infty$. These measures are not regular, but they are quasi-regular in the sense of \cite[Definition~3.22]{repst}. This suggests that we should be able to construct abelian envelopes for $\cC_k(0)$ and $\cC_k(\infty)$ by the methods of \cite{repst}. These envelopes should be non-semi-simple pre-Tannakian categories. We have not attempted to carry out the details here.

(d) Fix a field $k$ of characteristic~0. Let $\bT=(T_1, \ldots, T_r)$ be a tuple of trees, each of level $\le \ell$. By Corollary~\ref{cor:AB-isom}, we have $B^{\bT}_k(n) \cong A^{\bT}_k(n)$ for $n \ge 2 \ell$, and this algebra is the specialization of the algebra $A^{\bT}_{k[\tau]}(\tau)$ at $\tau=n$. Let $\bigoplus_{i=1}^r L_i^{\oplus m_i}$ be the simple decomposition of $\Vec_{\bT}$ in the ``generic'' category $\cC_{k(\tau)}(\tau)$. By the Tits deformation theorem, the $L_i$'s can be specialized to distinct simple objects $L_i(n)$ in $\cD_k(n)$, and $\bigoplus_{i=1}^r L_i(n)^{\oplus m_i}$ is the simple decomposition of $\Vec_{\bT}$ in $\cD_k(n)$. Moreover, the categorical dimension of $L_i(n)$ is the value at $n$ of the categorical dimension of $L_i$ (which is a rational function of $\tau$). This shows that the categories $\cD_k(n)$ exhibit the phenomenon of representation stability, just like the symmetric groups (see \cite{ChurchFarb}).

\section{Some examples} \label{s:example}

\def\TreeOne{\scalebox{.75}{
\begin{tikzpicture}[baseline=0pt]
\node[leaf,label={\tiny a}] (A) at (0,0) {};
\end{tikzpicture}}}

\def\TreeTwo{\scalebox{.75}{
\begin{tikzpicture}[baseline=0pt]
\node[leaf,label={\tiny a}] (A) at (0,0) {};
\node[leaf,label={\tiny b}] (B) at (1,0) {};
\draw (A)--(B);
\end{tikzpicture}}}

\def\TreeThree{\scalebox{.75}{
\begin{tikzpicture}[baseline=0pt]
\node[boron] (Y) at (1,0) {};
\node[leaf,label={\tiny a}] (A) at (0,0) {};
\node[leaf,label={\tiny b}] (B) at (1.766,.642) {};
\node[leaf,label={\tiny c}] (C) at (1.766,-.642) {};
\draw (Y)--(A);
\draw (Y)--(B);
\draw (Y)--(C);
\end{tikzpicture}}}

\def\TreeFour{\scalebox{.75}{
\begin{tikzpicture}[baseline=0pt]
\node[boron] (X) at (0,0) {};
\node[leaf,label={\tiny a}] (A) at (1,0) {};
\node[leaf,label=left:{\tiny b}] (B) at (0,1) {};
\node[leaf,label={\tiny c}] (C) at (-1,0) {};
\node[leaf,label=left:{\tiny d}] (D) at (0,-1) {};
\draw (X)--(A);
\draw (X)--(B);
\draw (X)--(C);
\draw (X)--(D);
\end{tikzpicture}}}

\def\TreeFive{\scalebox{.75}{
\begin{tikzpicture}[baseline=0pt]
\node[boron] (X) at (0,0) {};
\node[boron] (Y) at (1,0) {};
\node[leaf,label={\tiny a}] (A) at (-.766,.642) {};
\node[leaf,label={\tiny b}] (B) at (-.766,-.642) {};
\node[leaf,label={\tiny c}] (C) at (1.766,.642) {};
\node[leaf,label={\tiny d}] (D) at (1.766,-.642) {};
\draw (X)--(Y);
\draw (X)--(A);
\draw (X)--(B);
\draw (Y)--(C);
\draw (Y)--(D);
\end{tikzpicture}}}

\def\TreeSix{\scalebox{.75}{
\begin{tikzpicture}[baseline=0pt]
\node[boron] (X) at (0,0) {};
\node[leaf,label={\tiny a}] (A) at (.951,.309) {};
\node[leaf,label={\tiny b}] (B) at (0,1) {};
\node[leaf,label={\tiny c}] (C) at (-.951,.309) {};
\node[leaf,label={\tiny d}] (D) at (-.588,-.809) {};
\node[leaf,label={\tiny e}] (E) at (.588,-.809) {};
\draw (X)--(A);
\draw (X)--(B);
\draw (X)--(C);
\draw (X)--(D);
\draw (X)--(E);
\end{tikzpicture}}}

\def\TreeSeven{\scalebox{.75}{
\begin{tikzpicture}[baseline=0pt]
\node[boron] (X) at (0,0) {};
\node[boron] (Y) at (1,0) {};
\node[leaf,label={\tiny a}] (A) at (-.5,.866) {};
\node[leaf,label={\tiny b}] (B) at (-1,0) {};
\node[leaf,label={\tiny c}] (C) at (-.5,-.866) {};
\node[leaf,label={\tiny d}] (D) at (1.766,.642) {};
\node[leaf,label={\tiny e}] (E) at (1.766,-.642) {};
\draw (X)--(Y);
\draw (X)--(A);
\draw (X)--(B);
\draw (X)--(C);
\draw (Y)--(D);
\draw (Y)--(E);
\end{tikzpicture}}}

\def\TreeEight{\scalebox{.75}{
\begin{tikzpicture}[baseline=0pt]
\node[boron] (X1) at (0,0) {};
\node[boron] (X2) at (1,0) {};
\node[boron] (X3) at (2,0) {};
\node[leaf,label={\tiny a}] (A) at (-.766,.642) {};
\node[leaf,label={\tiny b}] (B) at (-.766,-.642) {};
\node[leaf,label=left:{\tiny c}] (C) at (1,1) {};
\node[leaf,label={\tiny d}] (D) at (2.766,.642) {};
\node[leaf,label={\tiny e}] (E) at (2.766,-.642) {};
\draw (X1)--(X2);
\draw (X2)--(X3);
\draw (X1)--(A);
\draw (X1)--(B);
\draw (X2)--(C);
\draw (X3)--(D);
\draw (X3)--(E);
\end{tikzpicture}}}

\def\TreeNine{\scalebox{.75}{
\begin{tikzpicture}[baseline=0pt]
\node[boron] (X) at (0,0) {};
\node[leaf,label={\tiny a}] (A) at (1,0) {};
\node[leaf,label={\tiny b}] (B) at (.5,.866) {};
\node[leaf,label={\tiny c}] (C) at (-.5,.866) {};
\node[leaf,label={\tiny d}] (D) at (-1,0) {};
\node[leaf,label={\tiny e}] (E) at (-.5,-.866) {};
\node[leaf,label={\tiny f}] (F) at (.5,-.866) {};
\draw (X)--(A);
\draw (X)--(B);
\draw (X)--(C);
\draw (X)--(D);
\draw (X)--(E);
\draw (X)--(F);
\end{tikzpicture}}}

\def\TreeTen{\scalebox{.75}{
\begin{tikzpicture}[baseline=0pt]
\node[boron] (X) at (0,0) {};
\node[boron] (Y) at (1,0) {};
\node[leaf,label={\tiny a}] (A) at (-.766,.642) {};
\node[leaf,label={\tiny b}] (B) at (-.766,-.642) {};
\node[leaf,label=right:{\tiny c}] (C) at (1.5,.866) {};
\node[leaf,label=right:{\tiny d}] (D) at (1.94,.342) {};
\node[leaf,label=right:{\tiny e}] (E) at (1.94,-.342) {};
\node[leaf,label=right:{\tiny f}] (F) at (1.5,-.866) {};
\draw (X)--(Y);
\draw (X)--(A);
\draw (X)--(B);
\draw (Y)--(C);
\draw (Y)--(D);
\draw (Y)--(E);
\draw (Y)--(F);
\end{tikzpicture}}}

\def\TreeEleven{\scalebox{.75}{
\begin{tikzpicture}[baseline=0pt]
\node[boron] (X) at (0,0) {};
\node[boron] (Y) at (1,0) {};
\node[leaf,label=left:{\tiny a}] (A) at (-.5,.866) {};
\node[leaf,label=left:{\tiny b}] (B) at (-1,0) {};
\node[leaf,label=left:{\tiny c}] (C) at (-.5,-.866) {};
\node[leaf,label=right:{\tiny d}] (D) at (1.5,.866) {};
\node[leaf,label=right:{\tiny e}] (E) at (2,0) {};
\node[leaf,label=right:{\tiny f}] (F) at (1.5,-.866) {};
\draw (X)--(Y);
\draw (X)--(A);
\draw (X)--(B);
\draw (X)--(C);
\draw (Y)--(D);
\draw (Y)--(E);
\draw (Y)--(F);
\end{tikzpicture}}}

\def\TreeTwelve{\scalebox{.75}{
\begin{tikzpicture}[baseline=0pt]
\node[boron] (X) at (0,0) {};
\node[boron] (Y) at (1,0) {};
\node[boron] (Z) at (2,0) {};
\node[leaf,label={\tiny a}] (A) at (-.766,.642) {};
\node[leaf,label={\tiny b}] (B) at (-.766,-.642) {};
\node[leaf,label=left:{\tiny c}] (C) at (1,1) {};
\node[leaf,label=left:{\tiny d}] (D) at (1,-1) {};
\node[leaf,label={\tiny e}] (E) at (2.766,.642) {};
\node[leaf,label={\tiny f}] (F) at (2.766,-.642) {};
\draw (X)--(Y);
\draw (Y)--(Z);
\draw (X)--(A);
\draw (X)--(B);
\draw (Y)--(C);
\draw (Y)--(D);
\draw (Z)--(E);
\draw (Z)--(F);
\end{tikzpicture}}}

\def\TreeThirteen{\scalebox{.75}{
\begin{tikzpicture}[baseline=0pt]
\node[boron] (X) at (0,0) {};
\node[boron] (Y) at (1,0) {};
\node[boron] (Z) at (2,0) {};
\node[leaf,label={\tiny a}] (A) at (-.5,.866) {};
\node[leaf,label={\tiny b}] (B) at (-1,0) {};
\node[leaf,label={\tiny c}] (C) at (-.5,-.866) {};
\node[leaf,label=left:{\tiny d}] (D) at (1,1) {};
\node[leaf,label={\tiny e}] (E) at (2.766,.642) {};
\node[leaf,label={\tiny f}] (F) at (2.766,-.642) {};
\draw (X)--(Y);
\draw (Y)--(Z);
\draw (X)--(A);
\draw (X)--(B);
\draw (X)--(C);
\draw (Y)--(D);
\draw (Z)--(E);
\draw (Z)--(F);
\end{tikzpicture}}}

\def\TreeFourteen{\scalebox{.75}{
\begin{tikzpicture}[baseline=0pt]
\node[boron] (X1) at (0,0) {};
\node[boron] (X2) at (1,0) {};
\node[boron] (X3) at (2,0) {};
\node[boron] (X4) at (3,0) {};
\node[leaf,label={\tiny a}] (A) at (-.766,.642) {};
\node[leaf,label={\tiny b}] (B) at (-.766,-.642) {};
\node[leaf,label=left:{\tiny c}] (C) at (1,1) {};
\node[leaf,label=left:{\tiny d}] (D) at (2,1) {};
\node[leaf,label={\tiny e}] (E) at (3.766,.642) {};
\node[leaf,label={\tiny f}] (F) at (3.766,-.642) {};
\draw (X1)--(X2);
\draw (X2)--(X3);
\draw (X3)--(X4);
\draw (X1)--(A);
\draw (X1)--(B);
\draw (X2)--(C);
\draw (X3)--(D);
\draw (X4)--(E);
\draw (X4)--(F);
\end{tikzpicture}}}

\def\TreeFifteen{\scalebox{.75}{
\begin{tikzpicture}[baseline=0pt]
\node[boron] (X1) at (0,0) {};
\node[boron] (X2) at (1,0) {};
\node[boron] (X3) at (2,0) {};
\node[boron] (X4) at (1,1) {};
\node[leaf,label={\tiny a}] (A) at (-.766,.642) {};
\node[leaf,label={\tiny b}] (B) at (-.766,-.642) {};
\node[leaf,label=left:{\tiny c}] (C) at (.357,1.766) {};
\node[leaf,label=right:{\tiny d}] (D) at (1.642,1.766) {};
\node[leaf,label={\tiny e}] (E) at (2.766,.642) {};
\node[leaf,label={\tiny f}] (F) at (2.766,-.642) {};
\draw (X1)--(X2);
\draw (X2)--(X3);
\draw (X2)--(X4);
\draw (X1)--(A);
\draw (X1)--(B);
\draw (X4)--(C);
\draw (X4)--(D);
\draw (X3)--(E);
\draw (X3)--(F);
\end{tikzpicture}}}

\begin{figure}
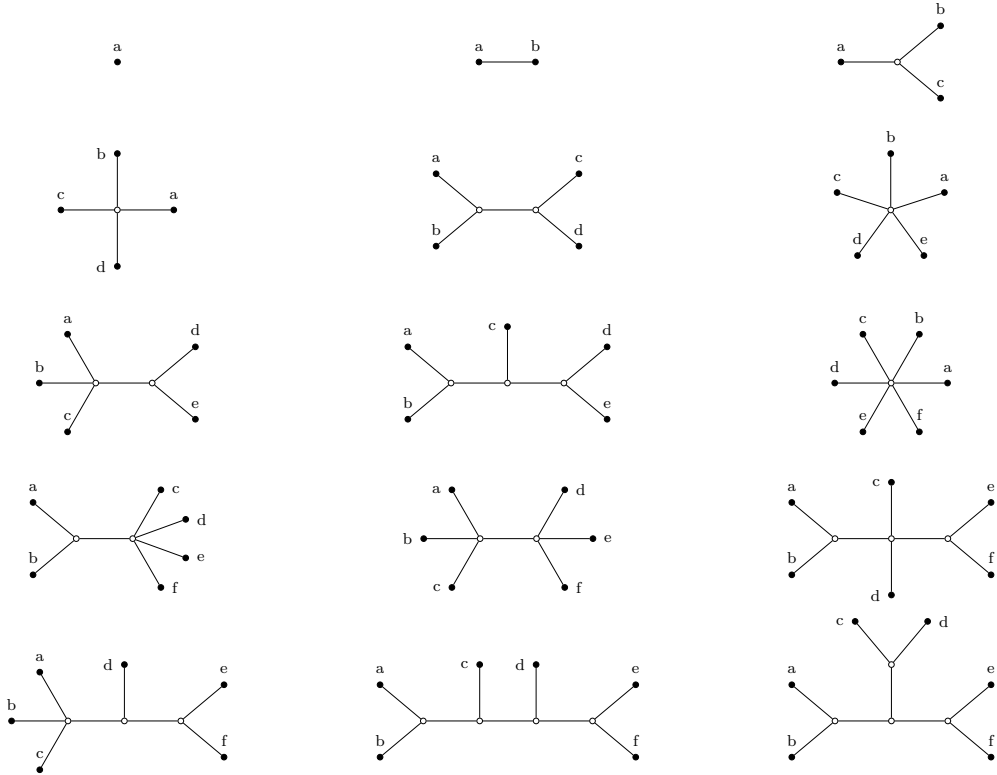

\def\arraystretch{4}
\setlength{\tabcolsep}{2em}
\begin{center}
\begin{tabular}{ccc}
\TreeOne & \TreeTwo & \TreeThree \\
\TreeFour & \TreeFive & \TreeSix \\
\TreeSeven & \TreeEight & \TreeNine \\
\TreeTen & \TreeEleven & \TreeTwelve \\
\TreeThirteen & \TreeFourteen & \TreeFifteen
\end{tabular}
\end{center}
\caption{Trees used in \S \ref{s:example}. We label these $T_1, \ldots, T_{15}$, starting at the top left and reading left to right.} \label{fig:trees}
\end{figure}

\subsection{Overview}

Fix an algebraically closed field $k$ of characteristic~0 and an element $t \in k \setminus \bN$ for the duration of \S \ref{s:example}. Thus the category $\cC_k(t)$ is semi-simple by Theorem~\ref{thm:nonsing}. In this section, we will find the simple decomposition of two basic objects in $\cC_k(t)$, and in doing so construct seven simple objects. All computations in this section take place in the category $\cC_k(t)$.

We will use the trees $T_1, \ldots, T_{15}$ depicted in Figure~\ref{fig:trees} throughout \S \ref{s:example}; we also let $T_0$ denote the empty tree. We often need to label these trees in new ways. To do so, we will write, e.g., $T_2(1,2)$ for the tree $T_2$ but with labels substituted as $a \to 1$ and $b \to 2$. Put
\begin{displaymath}
X=\Vec_{T_1(1)}, \qquad Y=\Vec_{T_2(1,2)}.
\end{displaymath}
Our aim is to find the simple decompositions of $X$ and $Y$.

\subsection{Preliminary remarks} \label{ss:prelim}

We first make a few general remarks that will be useful in analyzing our examples. Let $T$ be a tree. Recall that the arboreal algebra $A^T_k(t)=\End(\Vec_T)$ has a basis consisting of self-amalgamations of $T$. If $T$ has vertices $1, \ldots, n$, we will label the vertices of a self-amalgamation by $n+1, \ldots, 2n$ in this section (whereas we used primes in Example~\ref{ex:arboreal}). If $U$ is such a self-amalgamation then the endomorphism $\phi_U$ has a categorical trace, which we denote by $\utr(\phi_U)$. This, it turns out, is equal to $\mu(T)$ if $\phi_U$ is the identity, and otherwise vanishes. The categorical dimension of $\Vec_T$, denoted $\udim(\Vec_T)$, is the categorical trace of the identity, and thus equal to $\mu(T)$.

Suppose $U$ is a self-amalgamation of $T$. We define the \defn{transpose} of $U$, denoted $U^t$, to be the self-amalgamation where the two index sets are transposed, i.e., $1 \leftrightarrow n+1$, $2 \leftrightarrow n+2$, and so on. This operation induces a transpose operation on the algebra $\End(X)$, via $\phi_U^t=\phi_{U^t}$. This satisfies the usual property, i.e., $(ab)^t=b^t a^t$.

The algebra $\End(\Vec_T)$ carries a trace pairing, defined by $\langle a, b \rangle = \utr(a b)$. This pairing is non-degenerate, and described explicitly as follows. Suppose $a=\phi_U$ and $b=\phi_V$, where $U$ and $V$ are self-amalgamations of $T$. If $V=U^t$ then $\langle a, b \rangle = \mu(U)$. Otherwise, $\langle a, b \rangle = 0$. See \cite[Proposition~7.23]{repst} for details.

Suppose $a$ is an endomorphism of $\Vec_T$. We then have an expression $a = \sum_{i=1}^n \alpha_i \phi_{U_i}$, where the $U_i$'s are the various self-amalgamations, and the $\alpha_i$'s are constants. We have $\alpha_i \cdot \mu(U_i) = \langle a, \phi_{U_i^t} \rangle$. In particular, if $a=\phi_V \cdot \phi_W$ is a product of two basis vectors then $\alpha_i = \mu(U_i)^{-1} \utr(\phi_{U_i^t} \phi_V \phi_W)$. We can thus use traces of triple products to determine the structure constants in arboreal algebras.

We say a bit more about traces of triple products, as they will play an important role in this section. Let $U$, $V$, and $W$ be three self-amalgamations of $T$. It will be convenient to use three blocks of indices, namely, $1, \ldots, n$, and $n+1, \ldots, 2n$, and $2n+1, \ldots, 3n$. We use the first two blocks in $U$, the first and third blocks in $V$, and the second and third blocks in $W$. Let $S$ be the set of all trees on $1, \ldots, 3n$ such that the induced tree on the first and second block is $U$, on the first and third is $V$, and on the second and third is $W$. Then $\utr(\phi_U \phi_V \phi_W)$ is the sum of the measures of the trees in $S$.

Let $L_0=\bbone=\Vec_{T_0}$ be the trivial representation. There is always a unique amalgamation of $T_0$ and $T$, and so the mapping space $\Hom(\Vec_{T_0}, \Vec_T)$ is one-dimensional. It follows that $L_0$ occurs with multiplicity one in $\Vec_T$. The idempotent of $\End(\Vec_T)$ projecting onto the trivial representation is the sum of all basis vectors, divided by $\mu(T)$.

\subsection{The decomposition of $X$}

There are two self-amalgamations of $T_1(1)$, namely, $T_1(1/2)$ and $T_2(1,2)$. Here $T_1(1/2)$ is the tree with a single vertex labeled by both~1 and~2. We thus see that $\End(X)$ is a two-dimensional algebra. As the trivial representation $L_0$ has multiplicity one in $X$, it follows that $X$ decomposes as $X=L_0 \oplus L_1$, where $L_1$ is a non-trivial simple. We have
\begin{displaymath}
\udim(X)=\mu(X) = \frac{t}{t-1}, \qquad \udim(L_1)=\udim(X)-1 = \frac{1}{t-1}.
\end{displaymath}

\subsection{Initial comments on $Y$}

The space $\End(Y)$ is 10-dimensional, with basis:
\begin{align*}
a_1 &= T_2(1/3, 2/4) & a_6 &= T_3(1, 2/4, 6) \\
a_2 &= T_2(1/4, 2/3) & a_7 &= T_5(1, 2, 3, 4) \\
a_3 &= T_3(1/3, 2, 4) & a_8 &= T_5(1, 3, 2, 4) \\
a_4 &= T_3(1/4, 2, 3) & a_9 &= T_5(1, 4, 2, 3) \\
a_5 &= T_3(1, 2/3, 4) & a_{10} &= T_4(1, 2, 3, 4)
\end{align*}
The element $a_1$ is the identity element of the algebra $\End(Z)$. The element $a_2$ is a ``permutation matrix.'' Left multiplication by $a_2$ switches 3 and 4, while right multiplication by $a_2$ switches 1 and 2. We let $Y_+$ and $Y_-$ be the $+1$ and $-1$ eigenspaces of $a_2$ on $Y$. Thus $Y$ decomposes as $Y_+ \oplus Y_-$, and so we have a decomposition
\begin{displaymath}
\End(Y) = \End(Y_+) \oplus \End(Y_-) \oplus \Hom(Y_+, Y_-) \oplus \Hom(Y_-, Y_+).
\end{displaymath}
The space $\End(Y_-)$ is three dimensional, with basis
\begin{align*}
b_1 &= \tfrac{1}{2} ( a_1 - a_2 ) \\
b_2 &= \tfrac{1}{4} ( a_3 - a_4 - a_5 + a_6) \\
b_3 &= \tfrac{1}{2} ( a_8 - a_9 )
\end{align*}
while the space $\End(Y_+)$ is five dimensional, with basis
\begin{align*}
c_1 &= \tfrac{1}{2} (a_1+a_2) & c_4 &= \tfrac{1}{2} (a_8+a_9) \\
c_2 &= \tfrac{1}{4} (a_3+a_4+a_5+a_6) & c_5 &= a_{10} \\
c_3 &= a_7
\end{align*}
The spaces $\Hom(Y_+, Y_-)$ and $\Hom(Y_-, Y_+)$ are each one dimensional. We thus see that $Y_+$ and $Y_-$ have a single simple object in common, and that $Y_-$ is multiplicity-free.

The space $\Hom(X,Y)$ is three dimension, with basis
\begin{displaymath}
f_1=T_2(\ast/1,2), \qquad f_2=T_2(1,\ast/2), \qquad f_3=T_3(\ast,1,2).
\end{displaymath}
Here we use the index~$\ast$ on the source, and the indices~1 and~2 on the target. This space admits a $\bZ/2$ action, via the action of $a_2$ on $Y$. This action transposes $f_1$ and $f_2$, and fixes $f_3$. We thus see that $\Hom(X,Y_+)$ is 2-dimensional, with basis $f_1+f_2$ and $f_3$, while $\Hom(X,Y_-)$ is 1-dimensional, with basis $f_1-f_2$.

We know that $L_0$ occurs with multiplicity one in $Y$, and it is easily seen to occur in $Y_+$. From the dimension computations above, we see that $L_1$ occurs with multiplicity one in $Y_+$ and $Y_-$. Since $\End(Y_+)$ is 5-dimensional and $Y_+$ contains the two simples $L_0$ and $L_1$ with multiplicity one, it follows that $Y_+$ is multiplicity-free. We have already seen that $Y_-$ is multiplicity-free. We thus have
\begin{align*}
Y_+ &= L_0 \oplus L_1 \oplus L_2 \oplus L_3 \oplus L_4 \\
Y_- &= L_1 \oplus L_5 \oplus L_6
\end{align*}
where $L_0, \ldots, L_6$ are distinct simples, and $L_0$ and $L_1$ are the two simples already found. (Recall that since $\Hom(Y_+,Y_-)$ is 1-dimensional, there is only one common simple between $Y_+$ and $Y_-$, and we have found it to be $L_1$.)

\subsection{Decomposition of $Y_-$}

We have determined the decomposition of $Y_-$ in an abstract sense. We now aim to make this explicit. The following is our main result:

\begin{proposition} \label{prop:Yminus}
The elements
\begin{align*}
e_1 &= \frac{2(t-1)}{t}(b_1+b_2) \\
e_5 &= -\frac{(t-1)^2}{t}((t-2)b_1+2(t-1)b_2+tb_3) \\
e_6 &= (t-2)^2 b_1 + 2(t-1)(t-2) b_2 + (t-1)^2 b_3
\end{align*}
are the idempotents of $\End(Y_-)$ projecting on $L_1$, $L_6$, and $L_6$.
\end{proposition}

Note that prior to this proposition, we have not distinguished $L_5$ and $L_6$. Thus the proposition is more properly stated as: $e_1$, $e_5$, and $e_6$ are the primitive idempotents of $\End(Y_-)$, and $e_1$ projects onto $L_1$. Taking traces, we find:

\begin{corollary}
We have
\begin{displaymath}
\udim(L_5) = -\tfrac{1}{2} (t-2), \quad
\udim(L_6) = \frac{t(t-2)^2}{2(t-1)^2}.
\end{displaymath}
\end{corollary}

To prove the proposition, we show that $\End(Y_-)$ is generated by $b_3$, and determine the minimal polynomial of $b_3$; this will give an explicit description of $\End(Y_-)$, from which the proposition easily follows. The details are carried out in the following lemmas.

\begin{lemma} \label{lem:idemp}
$e_1$ is the projector onto $L_1$.
\end{lemma}

\begin{proof}
Let $f \colon Y \to X$ be the map corresponding to $T_2(1/*, 2)$ and let $f' \colon X \to Y$ be the map corresponding to $T_2(3/*, 4)$. Here we use coordinates $\{1,2\}$ on the first $Y$, $\{3,4\}$ on the second $Y$, and $*$ on $X$. Thus $e'=\tfrac{1}{4} (1-a_2)f' f(1-a_2)$ is an endomorphism of $Y_-$ that factors through $X$. It is thus necessarily a scalar multiple of the projector onto $L_1$.

To compute $f' f$ we must find all trees on $\{1,2,3,4,*\}$ that restrict to the trees giving $f$ and $f'$. There are two such trees, namely, $T_2(1/3/*, 2/4)$ and $T_3(1/3/*, 2, 4)$. To compute the actual composition, we delete $*$ and take the measure of the resulting inclusion. In both cases, this measure is just~1. Thus $f' f =a_1+a_3$, and so $e'=b_1+b_2$.

We thus see that $b_1+b_2$ of the projector onto $L_1$. We now determine the scalar. Write $(b_1+b_2)^2=\lambda \cdot (b_1+b_2)$. To compute $\lambda$, take traces. We have
\begin{displaymath}
\utr(b_1) = \frac{t}{2(t-1)^2}, \qquad
\utr(b_2) = 0, \qquad
\utr(b_2^2) = -\frac{t(t-2)}{4 (t-1)^3}.
\end{displaymath}
We thus find
\begin{displaymath}
\lambda = 1 +\utr(b_2^2)/\utr(b_1) = 1 - \frac{t-2}{2(t-1)} = \frac{t}{2(t-1)}.
\end{displaymath}
The result now follows, as $\lambda^{-1}(b_1+b_2)$ is the projector in question.
\end{proof}

\begin{lemma}
We have $(b_1+b_2)b_3=0$; equivalently, $b_2b_3=-b_3$.
\end{lemma}

\begin{proof}
Since $b_1+b_2$ is a primitive idempotent, we have $(b_1+b_2)b_3 = \lambda (b_1+b_2)$ for some scalar $\lambda$. Since $b_3$ is orthogonal to $b_1$ and $b_2$ under the trace pairing, and $\utr(b_1+b_2) \ne 0$, we find $\lambda=0$, as required.
\end{proof}

\begin{lemma}
We have
\begin{displaymath}
\utr(a_8^3)=\frac{t(t-2)^2(2t-3)^2}{(t-1)^6}.
\end{displaymath}
\end{lemma}

\begin{proof}
Recall that $a_8=T_5(1,3,2,4)$. Let $S$ be the set of all trees labeled by $1, \ldots, 6$ such that when we delete 5 and 6 we obtain $T_5(1,3,2,4)$, when we delete 3 and 4 we obtain $T_5(5,1,6,2)$, and when we delete 1 and 2 we obtain $T_5(3,5,4,6)$. Then $\utr(a_8^3)$ is the sum of the measures of all trees in $S$. The set $S$ has 16 elements. Three representative examples are
\begin{displaymath}
T_{11}(1,3,5,2,4,6), \quad T_{12}(1,3,5,2,4,6), \quad T_{13}(1,3,5,2,4,6).
\end{displaymath}
In each case, one can permute 1, 3, and 5 arbitrarily, and also 2, 4, and 6; moreover, one can place 2, 4, 6 before 1, 3, 5. This results in one element of type $T_{11}$, nine of type $T_{12}$, and six of type $T_{13}$. We thus find
\begin{displaymath}
\utr(a_8^3) = \mu(T_{11}) + 9 \mu(T_{12}) + 6 \mu(T_{13}).
\end{displaymath}
We have
\begin{displaymath}
\mu(T_{11}) = \frac{t (t-2)^2 (t-3)^2}{(t-1)^6}, \quad
\mu(T_{12}) = \frac{t (t-2)^4}{(t-1)^6}, \quad
\mu(T_{13}) = -\frac{t (t-2)^3 (t-3)}{(t-1)^6}
\end{displaymath}
This yields the stated formula.
\end{proof}

\begin{lemma}
We have
\begin{displaymath}
\utr(a_8^2 a_9)=\frac{2 t(t-2)^4}{(t-1)^6}.
\end{displaymath}
\end{lemma}

\begin{proof}
This is similar to the previous computation, but now we want to obtain $T_5(5,2,6,1)$ when we delete 3 and 4. The set $S$ has two elements, namely,
\begin{displaymath}
T_{15}(1,3,2,5,4,6), \qquad T_{15}(1,6,2,4,3,5).
\end{displaymath}
We thus see that the trace is $2 \mu(T_{15})$, and so the result follows.
\end{proof}

\begin{lemma}
We have
\begin{displaymath}
b_3^2 = \frac{(t-2)^2}{(t-1)^2} b_1 + \frac{2(t-2)}{t-1} b_2 + \frac{2t^2-4t+1}{(t-1)^2} b_3
\end{displaymath}
\end{lemma}

\begin{proof}
Write $b_3^2 = \alpha b_1+\beta b_2 + \gamma b_3$, where the coefficients are to be determined. Taking traces, we find $\alpha=\utr(b_3^2)/\utr(b_1)$, from which one easily obtains the formula for $\alpha$. Since $(b_1+b_2)b_3=0$, we find $(\alpha b_1+\beta b_2)(b_1+b_2)=0$, which allows one to compute $\beta$ in terms of $\alpha$. Finally, we have $\gamma = \utr(a_8 b_3^2)/\utr(a_8^2)$. We have
\begin{displaymath}
b_3^2 = \tfrac{1}{4} (a_8^2-a_8a_9-a_9a_8+a_9^2) = \tfrac{1}{4} (2a_8^2-a_8 a_9-a_9a_8),
\end{displaymath}
where we used the identity $a_9^2=a_8^2$, which follows since $a_9=a_2a_8=a_8a_2$. We thus see
\begin{displaymath}
\utr(a_8 b_3^2) = \tfrac{1}{2}( \utr(a_8^3) - \utr(a_8^2a_9) ),
\end{displaymath}
and the result follows from the previous two lemmas.
\end{proof}

\begin{lemma}
The elements 1, $b_3$, and $b_3^2$ form a basis for $\End(Y_-)$. The minimal polynomial of $b_3$ is $b_3 (b_3-1)(b_3-\xi)$, where $\xi=t(t-2)/(t-1)^2$.
\end{lemma}

\begin{proof}
Since $b_2$ has non-zero coefficient in $b_3^2$, it follows that 1, $b_3$, and $b_3^2$ form a basis for $\End(Y_-)$. Multiplying the formula for $b_3^2$ by $b_3$, and using $b_2b_3=-b_2$, we find
\begin{displaymath}
b_3^3=\alpha b_3 - \beta b_3 + \gamma b_3^2,
\end{displaymath}
where the coefficients on the right are as in the above proof. This yields the minimal polynomial for $b_3$. An elementary manipulation puts it into the stated form.
\end{proof}

\subsection{Decomposition $Y_+$}

We now aim to do for $Y_+$ what we have just done for $Y_-$. The following is the main result:

\begin{proposition} \label{prop:Yplus}
The elements
\begin{align*}
f_0 &= \frac{(t-1)^2}{t} (2c_1+4c_2+c_3+2c_4+c_5) \\
f_1 &= -\frac{2(t-1)}{t-2}(c_1+c_2 - \tfrac{1}{t-1}f_1) \\
f_2 &= \frac{2(t-1)}{t-2}(c_1-c_4-\tfrac{1}{2} c_5-\frac{t}{2(t-1)} f_1 ) \\
f_3 &= (t-2)^2 c_1 + 2(t-1)(t-2) c_2 + (t-1)^2 c_4 \\
f_4 &= -\frac{t-1}{(t-2)^2} \cdot ((t-3)f_4-(t-3)c_4-(t-2)c_5)
\end{align*}
of $\End(Y_+)$ are the idempotents projecting onto $L_0, \ldots, L_4$.
\end{proposition}

The same caveat as with Proposition~\ref{prop:Yminus} applies. Taking traces, we find:

\begin{corollary}
We have
\begin{displaymath}
\udim(L_2)=-\frac{t}{t-1}, \quad
\udim(L_3)=\frac{t(t-2)^2}{2(t-1)^2}, \quad
\udim(L_4)=-\frac{t(t-3)}{2(t-1)}.
\end{displaymath}
\end{corollary}

We now give a (somewhat abbreviated) explanation of the proof. First, the formula for $f_0$ follows from the general discussion in \S \ref{ss:prelim}. The formula for $f_1$ follows from a similar analysis as in Lemma~\ref{lem:idemp}.

We now use a trick to find $f_2$; this is not strictly necessary, but simplifies matters somewhat. Recall from \S \ref{ss:further}(a,b) that we can define $\cC_k(t)$ using ``enhanced trees.'' Let $X'=\Vec_{T_1'}$, where $T_1'$ is the enhanced tree with one node, and no other vertices; this is actually a summand of $\Vec_{T_3}$. We can analyze $\Vec_{X'}$ similar to how we analyzed $\Vec_X$; we find that it decomposes as $L_0 \oplus L_2$, where $L_2$ is a simple not isomorphic to $L_0$ or $L_1$. By analyzing maps $X' \to Y$, one can show that $L_2$ appears in $Y_+$. Computing the composition of maps $Y \to X' \to Y$ and projecting on $Y_+$ gives a scalar multiple of the projector onto $L_2$, and the scalar can be determined by taking traces, just as in Lemma~\ref{lem:idemp}. This yields the formula for $f_2$.

Here is the how we find the remaining two idempotents. First, note that these two idempotents are orthogonal to $f_0$, $f_1$, and $f_2$ under the trace pairing. It is straightforward to find a basis $u$ and $v$ for this orthogonal complement. Next, note that if $f$ is any primitive idempotent and $x$ is any element of $\End(Y_+)$ then $xf$ is a scalar multiple of $f$. Thus the remaining two idempotents are eigenvectors for the multiplication by $c_5$ map on $\operatorname{span}(u,v)$. We compute $c_5 u$ and $c_5 v$, using methods similar to the previous section, and then find the eigenvectors.

We can actually simplify this procedure a little with hindsight, since we already know the answers. First, it is a simple computation to check that $f_4$ and $f_5$ are orthogonal to the other idempotents under the trace pairing, and they are obviously linearly independent from each other. Next, write $c_5f_4=\alpha f_4+\beta f_5$. We can determine $\alpha$ and $\beta$ from $\utr(c_5 f_4)$ and $\utr(c_5^2 f_4)$. It is clear that $\utr(c_5f_4)=0$ since $c_5$ does not appear in $f_4$, and a non-trivial computation shows that $\utr(c_5^2 f_4)$ vanishes; thus $c_5f_4=0$. A similar procedure shows that $c_5 f_5 = 2 (t-1)^{-2} f_5$. Thus $f_4$ and $f_5$ are eigenvectors for multiplication by $c_5$, with distinct eigenvalues, and so must be scalar multiples of the remaining two idempotents. The scalar can be determined by the usual method (take traces).

We record some of the key intermediate computations here:

\begin{lemma}
We have the following:
\begin{align*}
\utr(c_5^2 c_2) &= \frac{2t(t-2)(t-3)}{(t-1)^5} \\
\utr(c_5^2 c_4) &= \frac{t(t-2)^2(t-3)(t-6)}{(t-1)^6} \\
\utr(c_5^3) &= - \frac{t(t-2)(t-3)(t-4)(t-5)}{(t-1)^6}
\end{align*}
\end{lemma}

\begin{proof}
We comment on each computation in turn.

(a) Note that $c_5=a_{10}$ is left and right invariant under $\bZ/2$, and so we have $\utr(c_5^2 c_2)=\utr(a_{10}^2 a_3)$. Let $S$ be the set of trees on $1, \ldots, 6$ that restrict to $T_4(1,2,3,4)$, $T_4(3,4,5,6)$, $T_3(1/5,2,6)$. This set contains two elements, namely, $T_6(1/5,2,3,4,6)$ and $T_7(1/5,3,4,2,6)$. We thus find that the trace is equal to $\mu(T_6)+\mu(T_7)$, which yields the stated formula.

(b) Again, we can use $a_8$ instead of $c_4$. Let $S$ be the set of trees on $1, \ldots, 6$ which restrict to $T_4(1,2,3,4)$, $T_4(3,4,5,6)$, and $T_5(1,5,2,6)$. This set has three members, namely, $T_{10}(1,5,2,3,4,6)$, $T_{10}(2,6,1,3,4,5)$, and $T_{12}(1,5,3,4,2,6)$. We thus find that the trace is equal to $2\mu(T_{10})+\mu(T_{12})$, which yields the stated formula.

(c) Let $S$ be the set of trees on $1, \ldots, 6$ which restrict to $T_4(1,2,3,4)$, $T_4(1,2,5,6)$, and $T_4(3,4,5,6)$. The set $S$ contains exactly one element, namely, $T_9(1,2,3,4,5,6)$. We thus see that $\utr(c_5^3)=\mu(T_9)$, which yields the stated formula.
\end{proof}

\end{document}